\documentclass[a4paper,10pt]{amsart}
\usepackage{tikz}
\usepackage{amsfonts,amsmath,amssymb}
\usepackage{color}
\usepackage{comment}
\usepackage{graphics}
\usepackage[hidelinks]{hyperref}
\usepackage{tikz,pgfplots}

\usepackage{listings}
\usepackage{mathtools}
\usepackage{esint}

\usepackage{algorithm}
\usepackage{algpseudocode}

\usepackage{todonotes}

\newtheorem{theorem}{Theorem}

\theoremstyle{definition}

\newtheorem{lemma}[theorem]{Lemma}
\theoremstyle{remark}
\newtheorem{remark}[theorem]{Remark}

\numberwithin{theorem}{section}
\numberwithin{equation}{section}
\numberwithin{table}{section}
\numberwithin{figure}{section}

\usepackage{caption}
\usepackage{subcaption}
\usepackage{graphics}

\newcommand{\supp}{\operatorname{supp}}
\newcommand{\calT}{\mathcal{T}}
\newcommand{\calE}{\mathcal{E}}

\newcommand{\N}{\mathbb{N}}

\newcommand{\nr}{n_\mathrm{res}}

\newcommand{\id}{\mathsf{id}}

\newcommand{\Nb}{\mathtt{N}}

\newcommand{\Dt}{D_\tau}

\newcommand{\LLL}{\lvert\!\lvert\!\lvert}
\newcommand{\RRR}{\rvert\!\rvert\!\rvert}

\definecolor{darkgreen}{rgb}{0.09, 0.45, 0.27}
\definecolor{myOrange}{rgb}{0.85000,0.32500,0.09800}
\definecolor{myBlue}{RGB}{30,144,255} 
\definecolor{myGreen}{RGB}{69,169,100} 
\definecolor{myRed}{RGB}{165,12,42} 
\definecolor{myOrange}{RGB}{225,92,22} 

\allowdisplaybreaks

\begin{document}
%
%
%====================================================================
%=========  Title / Contents
%====================================================================
\title[Localized implicit time stepping for the wave equation]
      {Localized implicit time stepping for the wave equation}
\author[D.~Gallistl \and R.~Maier]{Dietmar Gallistl \and Roland Maier}
\address[D.~Gallistl]{Institute of Mathematics, 
          Friedrich-Schiller-Universit\"at Jena, 
          Ernst-Abbe-Platz 2, 07743 Jena, Germany}
\email{dietmar.gallistl@uni-jena.de}
\address[R.~Maier]{Institute for Applied and Numerical Mathematics,
Karlsruhe Institute of Technology,
Englerstr.~2,
76131 Karlsruhe, Germany}
\email{roland.maier@kit.edu}
\thanks{D.~Gallistl is supported by the European Research Council
        (ERC Starting Grant \emph{DAFNE}, agreement ID 891734).}
\date{\today}

\begin{abstract}
This work proposes a discretization of the acoustic wave equation 
with possibly oscillatory coefficients based on
a superposition of discrete solutions to spatially localized 
subproblems computed with an implicit time discretization. 
Based on exponentially decaying entries of the global system 
matrices and an appropriate partition of unity, it is proved
that the superposition of localized solutions is appropriately
close to the solution of the (global) implicit scheme.
It is thereby justified that the localized (and especially parallel) 
computation on multiple overlapping subdomains is reasonable.
Moreover, a re-start is introduced after a certain amount of time steps
to maintain a moderate overlap of the subdomains. 
Overall, the approach may be understood as a domain decomposition 
strategy in space on successive short time intervals that completely 
avoids inner iterations. 
Numerical examples are presented.
\end{abstract}


\maketitle
\section{Introduction}\label{s:intro}

We consider the acoustic wave equation
\begin{subequations}\label{eq:exactSol}
	\begin{align}
		\partial_t^2 u - \nabla\cdot(A\nabla u) 
		= f \quad\text{in }\Omega\times(0,t_\mathrm{fin})
	\end{align}
	in an open, bounded, connected Lipschitz polytope
	$\Omega\subset\mathbb R^d$ in $d$ dimensions until some final time 
	$t_\mathrm{fin} > 0$,
	subject to the boundary condition
	\begin{align}
		u = 0 \quad\text{on }\partial\Omega\times(0,t_\mathrm{fin})
	\end{align}
	and the initial conditions
	\begin{align}
		u|_{t=0 } = u_0 
		\quad\text{and}\quad
		\partial_tu|_{t=0} = v_0 \qquad \text{in }\Omega.
	\end{align}
\end{subequations}
Here $u$ is the (unknown) wave field,
$A \in L^{\infty}(\Omega)$ is a positive and bounded 
coefficient that satisfies $\alpha \leq A(x) \leq \beta$
almost everywhere for 
given positive constants $\alpha$ and~$\beta$;
$u_0$ and $v_0$ are the initial data
and $f$ is a given source term.
Details can be found in Section~\ref{s:prelim}. 
Note that, more generally, a matrix-valued coefficient could be considered as well, but we stick to a scalar-valued coefficient for the ease of presentation.
Explicit time integration schemes 
-- such as the \emph{leapfrog method} as analyzed, e.g., in~\cite{Jol03,Chr09} --
for the numerical solution
of \eqref{eq:exactSol} are easy to implement and computationally
cheap in every single time step, in particular if mass lumping is used.
However, the time step size is typically limited by the CFL condition
that bounds the time step size by the spatial mesh parameter.
The advantage of implicit methods -- as, e.g., analyzed in~\cite{Dup73,Bak76} -- is that such a restriction
can be avoided, at the expense of solving a more involved linear system in each
time step.
While the propagation speed of the solution to~\eqref{eq:exactSol}
is finite, the support of the discrete solution defined by
an implicit scheme will usually equal $\overline\Omega$ after the first
time step. 
This suggests that, although the discrete information outside the
cone of propagation may be essential for the approximation to
fulfill some discrete conservation of energy,
it may be of minor importance in terms of approximation of 
the function $u$.
Given locally supported initial data and right-hand side,
the discrete functions computed from a global implicit method
or the same method restricted to a subdomain that includes
the physical cone of propagation, should therefore have 
comparable approximation properties.
The approach proposed in this work is based on this reasoning.
Although the matrix inversion related to each time step
of an implicit scheme transports information globally over
$\Omega$, relevant information decays fast
and can be captured by solving a system over a smaller sub-domain,
which implies less computational cost.
Global initial data and sources can be localized through a partition
of unity and the discrete solution can be defined by superposition
of the solutions to local subproblems that can be solved in
parallel.

In this paper, we work out this approach for the Crank--Nicolson
discretization of~\eqref{eq:exactSol} with first-order finite 
elements in space. 
We first present the classical (global) scheme and some preliminaries in Section~\ref{s:prelim}. 
We then quantify the decay of information and thus
the localizability of the problem in the energy norm and provide
an error estimate between the Crank--Nicolson solution and a
localized version (Section~\ref{s:decloc}). 
We then define the superposition scheme (referred to as \emph{local superposition method})
and provide an error estimate in Section~\ref{s:LSM}. 
Finally, numerical experiments are presented (Section~\ref{s:numerics}).

We emphasize that the general aim of our strategy is to avoid global computations at the cost of some limited overlap. 
This is particularly important if very fine discretization parameters are considered (e.g., in the context of multiscale problems with highly oscillatory coefficients), where global computations quickly become unfeasible.
Other approaches that address this issue are \emph{multiscale methods}, as considered, e.g., in~\cite{AbdG11,AbdH17,PetS17,OwhZ08,MaiP19} for the wave equation. 
These methods operate on a coarse (spatial) scale and only locally incorporate fine information into a suitable low-dimensional trial space, which needs to be built before the actual simulation. 
In~\cite{GeeM23}, such ideas are combined with a coefficient-adapted lumping strategy to achieve a fully explicit multiscale method. 
Note that multiscale methods typically use fine-scale information to
provide approximation properties with respect to the coarse scale,
whereas we use the coarse scale to make the computation of the fine-scale
solution more efficient.

Conceptually, our approach has some connections to classical \emph{domain decomposition} strategies; see, e.g., \cite{GanHN03,GanH05,GanKM21} for works in the context of the wave equation. 
However, compared to such methods we do not require multiple iterations and completely avoid the necessity to properly define transmission conditions between the subdomains at the cost of larger overlaps.  
Domain splitting ideas are also used in connection with locally implicit or local time-stepping schemes, see~\cite{DiaG09,GroMS21,CarH22,ChaI21}. From a practical point of view, our approach easily allows for a combination of different schemes in various subdomains as well, without the need for an appropriate coupling.
Finally, we would like to mention that the general idea of exploiting the fact that waves stay within the physical cone of propagation is also used in connection with so-called \emph{tent-pitching methods}, see~\cite{Ric94,FalR99,EriGSU05}.

\subsection*{Notation}
Throughout this work, $C > 0 $ denotes a generic constant that is 
independent of the scales $H,\varepsilon,h$, and $\vartheta$ but might 
depend on the dimension $d$ and the parameters $\alpha$ and $\beta$.
The value of $C$ might change from line to line in the estimates. 
Further, we write $ \theta \lesssim \eta$ if $\theta \leq C \eta$ 
and $\theta \eqsim \eta$ if 
$\theta \lesssim \eta\lesssim \theta$.
For the ease of notation, we assume that the diameter of the 
domain satisfies $\operatorname{diam}(\Omega)\eqsim 1$.
Finally, for $\omega \subseteq \Omega$ we use the notation $(\cdot,\cdot)_\omega := (\cdot,\cdot)_{L^2(\omega)}$, $a_\omega(\cdot,\cdot) := (A\nabla \cdot,\nabla \cdot)_\omega$, $\|\cdot\|_\omega := \|\cdot\|_{L^2(\omega)}$, as well as $\|\cdot\|_{a,\omega}^2:= a_\omega(\cdot,\cdot)$. Further, we abbreviate $(\cdot,\cdot) := (\cdot,\cdot)_\Omega$, $a(\cdot,\cdot) := a_\Omega(\cdot,\cdot)$, $\|\cdot\| := \|\cdot\|_\Omega$, and $\|\cdot\|_{a} := \|v\|_{a,\Omega}$.

\section{Preliminaries}
\label{s:prelim}
\subsection{Model problem}

Assuming that 
$f\in L^\infty(0,t_\mathrm{fin};L^2(\Omega))$, $u_0 \in H^1_0(\Omega)$, and $v_0 \in L^2(\Omega)$, 
the weak formulation of~\eqref{eq:exactSol} reads: find $u\in L^\infty(0,t_\mathrm{fin};H^1_0(\Omega))$, with $\partial_t u\in L^\infty(0,t_\mathrm{fin};L^2(\Omega))$ and $\partial_t^2u\in L^\infty(0,t_\mathrm{fin};H^{-1}(\Omega))$, such that $u|_{t=0}= u_0$, $\partial_tu|_{t=0} = v_0$ and 
\begin{align*}
	\langle \partial_t^2u, w \rangle + a(u,w) 
	= (f,w)
	\quad\text{for all }w\in H_0^1(\Omega)\text{ and a.e. }t\in(0,t_\mathrm{fin})
\end{align*}
with the duality pairing $\langle\cdot,\cdot\rangle$ of $H^{-1}(\Omega)$ with  $H^1_0(\Omega)$. 
Well-posedness of this problem follows directly by \cite[Ch.~3, Thm.~8.1 \& Rem.~8.2]{LioM72}. If also $\partial_tf\in L^\infty(0,t_\mathrm{fin};L^2(\Omega))$, $v_0 \in H^1_0(\Omega)$, and $f|_{t=0} + \nabla\cdot (A\nabla u_0) \in L^2(\Omega)$, we have $\partial_t^2u\in L^{\infty}(0,t_\mathrm{fin};L^2(\Omega))$ and can replace the duality pairing 
by the $L^2$-inner product.

\subsection{Classical discretization} 

In order to solve the model problem computationally, we discretize in space using the first-order finite element space $V_h$, which is a subspace of $H^1_0(\Omega)$ consisting of piecewise polynomials with coordinate degree at most one on a regular and quasi-uniform mesh $\calT_h$ with characteristic mesh size $h$. Further, for $\omega \subseteq \Omega$ we set $V_h(\omega) := \{v_h \in V_h \;\colon\;\supp(v_h)\subseteq \omega\}$. 
For the ease of notation, we assume $Q_1$ finite elements over parallelepipedal
meshes, but the arguments of this paper are valid for $P_1$ finite elements over
simplicial triangulations as well.
In time, we choose a \emph{Crank--Nicolson scheme} with time step size $\tau$.
It involves the second-order centered difference quotient
\begin{equation*}
	\hat\partial^2_\tau z_h^n:=\tau^{-2}(z_h^{n+1} - 2 z_h^n + z_h^{n-1})
\end{equation*}
in time and the averages
\begin{equation*}
	\widehat{z_h^n}:=\frac14 (z_h^{n+1} + 2 z_h^n + z_h^{n-1})
\end{equation*}
of given functions $(z_h^n)_n$.
The fully discrete scheme seeks $(u_h^n)_n$ with $u_h^n \in V_h$ such that
\begin{equation}\label{eq:CN}
	(\hat\partial^2_\tau u_h^n,v_h) 
	+
	a(\widehat{u_h^n},v_h) = (\widehat{f_h^{n}},v_h)
	\quad\text{for all }v_h \in V_h
\end{equation}
given appropriate initial conditions $u_h^1,\,u_h^{0} \in V_h$. 
Here, $f_h^k$ is an approximation of $f(k \tau)$ in $V_h$, e.g., its $L^2$-projection.
We emphasize that the scheme~\eqref{eq:CN} is obtained when the classical Crank--Nicolson scheme, which traces back to~\cite{CraN47}, is employed for the first-order formulation of the wave equation and the second variable is eliminated afterwards. Therefore, we use the term \emph{Crank--Nicolson scheme} here as well.
Using the mass matrix $M_h$ and stiffness matrix $S_h$, the scheme can be written as a system of linear equations as follows,
\begin{equation*}
	(M_h + \tfrac{\tau^2}{4} S_h)u_h^{n+1} = \tau^2 M_h\widehat{f_h^n}
	+ M_h(2u_h^n - u_h^{n-1}) - \tfrac{\tau^2}{4} S_h(2 u_h^n + u_h^{n-1}),
\end{equation*}
where we tacitly identify $u_h$ with its coefficient vector with respect
to a spatial basis.
We will use the conventions
\begin{equation} \label{eq:halfstep}
	\Dt z_h^{n+1/2} := \tfrac{z_h^{n+1} - z_h^n}{\tau},\qquad z_h^{n+1/2} := \tfrac{z_h^{n+1} + z_h^n}{2}
	\quad\text{for sequences } (z_h^n)_n.
\end{equation} 
We define the discrete energy norm
$$
\| v \|_{\mathcal E,\omega}= \left (\| \Dt v \|_\omega^2 + \|v\|_{a,\omega}^2\right)^{1/2}
$$
where we drop the domain dependence in the notation if 
$\omega=\Omega$,
and the discrete energy
\begin{equation*}
	\calE_{h,\tau}^{n + 1/2} :=
	\frac12 \| u_h^{n+1/2} \|_{\mathcal E}^2 .
\end{equation*}
We have the following result concerning the energy conservation of the scheme. 
\begin{theorem}[Energy conservation of the Crank--Nicolson method]\label{thm:energy}
	If $f \equiv 0$, we have energy conservation in the sense that
	\begin{equation*}
		\calE_{h,\tau}^{n+1/2} = \calE_{h,\tau}^{1/2}.
	\end{equation*}
	If $f \neq 0$, it holds that
	\begin{equation}\label{eq:stab}
		\sqrt{\calE_{h,\tau}^{n+1/2}} \leq \sqrt{\calE_{h,\tau}^{1/2}} + \sum_{j=1}^n \frac{\tau}{\sqrt{2}} \big\|\widehat{f_h^{j}}\big\|_{L^2(\Omega)}.
	\end{equation}
\end{theorem}

\begin{proof}
	The result follows directly with the choice 
	\begin{equation*} 
		v_h = u_h^{n+1} - u_h^{n-1} = (u_h^{n+1} - u_h^{n}) + (u_h^{n} - u_h^{n-1})
	\end{equation*} 
	as specific test function in~\eqref{eq:CN} after some minor algebraic modifications.
	A detailed proof is for instance provided in~\cite[Thm.~3.3.4]{Lon17}.
\end{proof}

\section{Decay and localization}\label{s:decloc}

In this section, we investigate the behavior of a discrete solution corresponding to local data. In particular, we investigate decay properties and motivate a localization of the discrete solution.

\subsection{Decaying solutions} In every time step, the system~\eqref{eq:CN} seeks the solution $u_h^{n+1} \in V_h$ to an equation characterized on the left-hand side by the bilinear form
\begin{equation*}
	\mathcal K(z_h,v_h):= (z_h,v_h)_{L^2(\Omega)} + \tfrac{\tau^2}{4}a(z_h,v_h).
\end{equation*}
For $h \approx \tau$, the inverse of the system matrix 
(with respect to the Lagrange basis) of this bilinear form
has exponentially decaying entries away from the diagonal (cf.~Figure~\ref{fig:sysM}). Therefore, \emph{relevant} information is only propagated by a certain amount from one time step to the subsequent one, similarly to the physical propagation of a (local) wave.
\begin{figure}
	\centering
	\includegraphics[width=.48\textwidth]{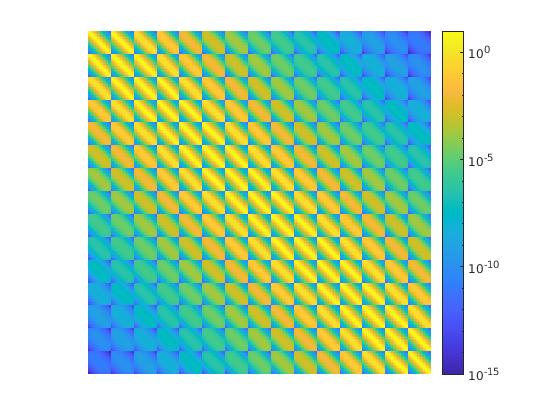}  
	\hfill
	\includegraphics[width=.48\textwidth]{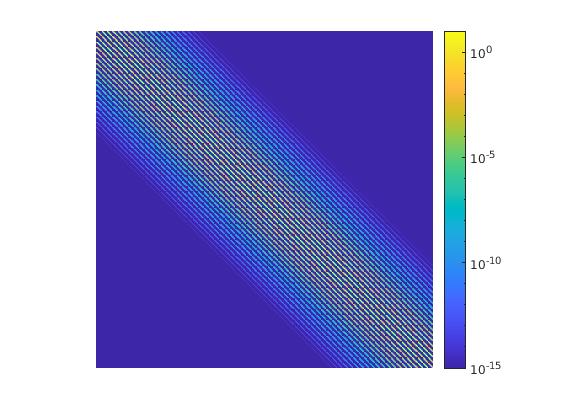}
	\caption%
	{Values of the inverse system matrix of the Crank--Nicolson scheme in two dimensions on a uniform and lexicographically ordered mesh with logarithmic color coding; mesh size (and time step) $h = \tau = 2^{-4}$ (left) and $h= \tau = 2^{-6}$ (right).} 
	\label{fig:sysM}
\end{figure}
To make this observation more rigorous, we show an appropriate decay estimate. Therefore, we require a special norm that is associated to the left-hand side of the  Crank--Nicolson scheme and reads, for any $\omega \subseteq \Omega$, 
\begin{equation*}
	\LLL v_h \RRR^2_{\omega} := 
	\mathcal K_\omega(v_h,v_h)
	=
	\|v_h\|_{L^2(\omega)}^2 + \tfrac{\tau^2}{4} \|v_h\|_{a,\omega}^2, \quad v_h \in V_h,
\end{equation*}
and abbreviate $\LLL v_h\RRR := \LLL v_h\RRR_{\Omega}$.
Further, we require the concept of element patches around a subset $\omega \subseteq \Omega$.  We define
\begin{equation*}
	\Nb_\ell(\omega) := \Nb(\Nb_{\ell-1}(\omega)),\,\ell \geq 2,
	\qquad 
	\Nb_1(\omega) = \Nb(\omega) := \operatorname{int}\bigcup \bigl\{\overline{K} \in \calT_h\,\vert\, \overline{\omega} \,\cap\, \overline{K}\neq \emptyset\bigl\}.
\end{equation*}
Further, we set $\Nb_0(\omega) := \operatorname{int}\omega$. 
We can now state the following result.
\begin{lemma}[Decaying discrete solution]\label{lem:decay}
	Let $\omega \subseteq \Omega$ be a union of elements such that $\supp f \subseteq \omega, \supp u_h^1 \subseteq \omega, \supp u_h^0 \subseteq \omega$. 
	For any $n \in \N$ and $\ell \in \N$, we have
	\begin{equation}\label{eq:dec}
		\LLL u_h^{n+1/2}\RRR_{\Omega \setminus \Nb_{\ell}(\omega)} 
		\leq (\ell+1)^{n/2}\gamma^{\ell} \max_{k \leq n}\LLL u_h^{k+1/2}\RRR
	\end{equation}
	with 
	$$
	0<\gamma = \sqrt{(C_{\tau,h} + \tfrac12)/(1+C_{\tau,h})} < 1
	$$
	for a constant 
	$C_{\tau,h} \eqsim \beta\alpha^{-1}(\tau/h + h/\tau)$ 
	that also depends on $d$. 
\end{lemma}

\begin{proof}
	The proof adopts the main ideas of decay proofs as used in the context of the multiscale method known as \emph{localized orthogonal decomposition}, see, e.g., \cite{MalP14,MalP20}. 
	For $n=0$, a related result was proved to show the decay of the Green's function for linear Schr\"odinger operators in an energy norm that is similar to our norm $\LLL\cdot\RRR$, see~\cite{AltHP20}.
	
	For any $\ell$, we define the cutoff function 
	$\eta_\ell \in V_h$ 
	with values $0 \leq \eta_\ell \leq 1$ and
	the following properties,
	\begin{equation}\label{eq:eta}
		\eta_\ell =
		\begin{cases}
			0&\text{in }\Nb_{{\ell-1}}(\omega) \\
			1&\text{in }\Omega\setminus\Nb_{{\ell}}(\omega)
		\end{cases}
		\qquad\text{and}\qquad
		\|\nabla\eta_\ell\|_{L^\infty(\Omega)} \leq 2^{(d-1)/2}\, h^{-1}.
	\end{equation}
	Further, we require the nodal interpolation operator $I_h \colon C^0(K) \to V_h$ onto $Q_1$ finite elements. 
	Standard interpolation estimates and the inverse inequality
	prove for any $K \in \calT_h$ and all $q \in Q_2(K)$ 
	(the space of biquadratic polynomial functions over $K$) 
	that
	\begin{equation}\label{eq:nodalInt}
		\|(\id - I_h)q\|_{a,K} 
		\leq C\beta^{1/2}\,h\, \|\nabla^2 q\|_{K}
		\leq C\beta^{1/2}\alpha^{-1/2}\| q\|_{a,K}. 
	\end{equation}
	The triangle inequality thus implies the stability property
	\begin{equation}\label{eq:H1cutoff}
		\|I_hq\|_{a,K} \leq \|q\|_{a,K} + \| (\id - I_h)q\|_{a,K}
		\leq (1+ C\beta^{1/2}\alpha^{-1/2})\| q\|_{a,K}
	\end{equation}
	for any $q \in Q_2(K)$.
	Using $\eta_\ell$ and $I_h$, we estimate for any $z_h \in V_h$
	\begin{align} \label{eq:proofDec1}
		\begin{aligned}
			\LLL z_h\RRR_{\Omega \setminus \Nb_\ell(\omega)}^2
			& = \mathcal K_{\Omega \setminus \Nb_\ell(\omega)}(z_h,z_h)
			\\
			&
			\leq (z_h,z_h \eta_\ell) + \tfrac{\tau^2}{4}\big(A \nabla z_h,\eta_\ell\nabla z_h\big)
			\\
			&
			= (z_h,z_h \eta_\ell) + \tfrac{\tau^2}{4}a(z_h,\eta_\ell z_h) - \tfrac{\tau^2}{4} \big(A \nabla z_h,z_h\nabla \eta_\ell\big)
			\\
			&
			=
			\mathcal K (z_h,I_h(z_h \eta_\ell))
			- \tfrac{\tau^2}{4} \big(A \nabla z_h,z_h\nabla \eta_\ell \big)
			+ \mathcal K (z_h,(\id-I_h)(z_h \eta_\ell)).
		\end{aligned}
	\end{align}
	Using that the support of $\nabla\eta_\ell$ is the closure of 
	$R_\ell:=\Nb_{\ell}(\omega) \setminus \Nb_{\ell-1}(\omega)$,
	the bound from \eqref{eq:eta} on $\nabla\eta_\ell$,
	and a weighted Young's inequality with weight
	$\delta = 2^{(3-d)/2} \tau^{-1}h\beta^{-1/2}$ 
	we obtain
	\begin{equation}\label{eq:proofDecTerm2}
		\begin{aligned}
			\big|\tfrac{\tau^2}{4} \big(A \nabla z_h,z_h\nabla \eta_\ell\big)\big|
			&
			\leq \tfrac1{2\delta} \tfrac{\tau^2}{4}\| z_h\|^2_{a,R_\ell} + \tfrac{\beta\delta}2 2^{(d-1)}\, h^{-2} \tfrac{\tau^2}{4} \|z_h\|_{R_\ell}^2
			\\
			&
			\leq 2^{(d-3)/2} h^{-1}\tau \beta^{1/2} \LLL z_h\RRR_{R_\ell}^2.
		\end{aligned}
	\end{equation}
	For the last term in~\eqref{eq:proofDec1}, we employ the fact that 
	$\supp ((\id - I_h)(z_h \eta_\ell)) \subseteq R_\ell$, a classical $L^2$-interpolation result of $I_h$, as well as~\eqref{eq:nodalInt}. 
	This leads to
	\begin{align}
		\mathcal K(z_h,(\id-I_h)(z_h \eta_\ell))
		& =\mathcal K_{R_\ell}(z_h,(\id-I_h)(z_h \eta_\ell)) \notag
		\\
		&
		= (z_h,(\id-I_h)(z_h \eta_\ell))_{R_\ell} 
		+ \tfrac{\tau^2}{4} a_{R_\ell}(z_h,(\id - I_h)(\eta_\ell z_h)) \label{eq:proofDec2}
		\\
		&
		\leq \big(C\alpha^{-1/2} h\,\|z_h\|_{R_\ell} 
		+ C \beta^{1/2}\alpha^{-1/2}\tfrac{\tau^2}{4} \|z_h\|_{a,R_\ell}\big) \|z_h\eta_\ell\|_{a,R_\ell}. \notag
	\end{align}
	Further, we have from the product rule and the bound from \eqref{eq:eta}
	\begin{equation}\label{eq:stabIh}
		\| z_h \eta_\ell \|_{a,R_\ell} 
		\leq Ch^{-1}\beta^{1/2} \|z_h\|_{R_\ell} + \| z_h\|_{a,R_\ell}.
	\end{equation}
	Going back to~\eqref{eq:proofDec2}, we arrive at 
	\begin{equation}\label{eq:proofDec3} 
		\mathcal K(z_h,(\id-I_h)(z_h \eta_\ell))
		\leq 
		C(\beta/\alpha)^{1/2} \LLL z_h\RRR_{R_\ell}^2
		+ C\beta\alpha^{-1/2}
		\big( h + \tfrac{\tau^2}{4h} \big) \|z_h\|_{R_\ell} \|z_h\|_{a,R_\ell}.
	\end{equation}
	A weighted Young's inequality with $\delta=\tau/2$ gives
	\begin{align*}
		\big( h + \tfrac{\tau^2}{4h} \big) \|z_h\|_{R_\ell} \|z_h\|_{a,R_\ell}
		\leq
		\big( h + \tfrac{\tau^2}{4h} \big) 
		\big( \frac{1}{2\delta} \|z_h\|_{R_\ell}^2
		+
		\frac{\delta}{2}  \|z_h\|_{a,R_\ell}^2 \big)
		=
		\big( \tfrac{h}{\tau} + \tfrac{\tau}{4h}\big)
		\LLL z_h\RRR_{R_\ell} ^2
		.
	\end{align*}
	The combination with \eqref{eq:proofDec3} shows that
	there exists a constant $C>0$ such that
	\begin{equation}\label{eq:proofDecTerm3}
		\mathcal K(z_h,(\id-I_h)(z_h \eta_\ell))
		\leq C \beta\alpha^{-1/2}
		\big( \tfrac{h}{\tau} + \tfrac{\tau}{4h}\big) \LLL z_h\RRR_{R_\ell}^2.
	\end{equation}
	Altogether, inserting~\eqref{eq:proofDecTerm2} and~\eqref{eq:proofDecTerm3} into~\eqref{eq:proofDec1}, we obtain
	\begin{equation}\label{eq:proofDec4}
		\LLL z_h\RRR_{\Omega \setminus \Nb_\ell(\omega)}^2 
		\leq 
		C_1 \LLL z_h\RRR_{R_\ell}^2 
		+ \mathcal K(z_h,I_h(z_h \eta_\ell))
	\end{equation}
	with a constant $C_1 \eqsim \beta\alpha^{-1/2}(\tau/h+h/\tau)$.
	
	We now turn to specific choices of $z_h \in V_h$ and prove the decay property~\eqref{eq:dec} by induction over $n$. For $n = 0$ and $z_h = u_h^{1/2}$, we have
	\begin{equation*}
		\mathcal K (u_h^{1/2},I_h(u_h^{1/2} \eta_\ell))= 0
	\end{equation*}
	since, by assumption, the supports of $u^1_h$ and $u^0_h$ have a trivial
	intersection with the support of $\eta_\ell$ and so with the support of 
	$I_h(\eta_\ell u_h^{1/2})$.
	Employing the elementary set identity
	$$
	R_\ell = (\Omega\setminus \Nb_{\ell-1}(\omega))\setminus (\Omega\setminus \Nb_{\ell}(\omega)),
	$$
	we therefore obtain from~\eqref{eq:proofDec4}
	\begin{equation*}
		\LLL u_h^{1/2}\RRR_{\Omega \setminus \Nb_\ell(\omega)}^2 
		\leq C_{1} \LLL u_h^{1/2}\RRR_{\Omega \setminus \Nb_{\ell-1}(\omega)}^2 
		- 
		C_{1} \LLL u_h^{1/2}\RRR_{\Omega \setminus \Nb_\ell(\omega)}^2
	\end{equation*}
	and thus with $\delta := \sqrt{C_{1}/(1+C_{1})} < 1$
	\begin{equation*}
		\LLL u_h^{1/2}\RRR_{\Omega \setminus \Nb_\ell(\omega)}^2 
		\leq \delta^{2\ell} \LLL u_h^{1/2}\RRR^2. 
	\end{equation*}
	In particular, \eqref{eq:dec} holds for $n = 0$ and any $\delta \leq \gamma < 1$.%\medskip
	
	Next, we consider the case $n=1$ with $z_h = u_h^{3/2}$.
	We use $u_h^{3/2}=\tfrac12u_h^{2}+\tfrac12u_h^{1}$
	and an argument similar to that for the case $n=0$ above
	together with $\supp(f_h) \subseteq \omega$
	to estimate the last term on the right-hand side of~\eqref{eq:proofDec4}
	(where we abbreviate $y:=I_h(u_h^{3/2} \eta_\ell)$)
	as follows
	\begin{align*}
		\mathcal K (u_h^{3/2},y)
		& =
		\tfrac12 \mathcal K(u_h^2,y)
		= \tfrac{\tau^2}{2} (\widehat{f_h^1},y)
		+ \tfrac12(2 u_h^1 - u_h^0,y) - \tfrac12\tfrac{\tau^2}{4}a(2u_h^1 + u_h^0, y)
		= 
		0.
	\end{align*}
	As above, we therefore get 
	$
	\LLL u_h^{3/2}\RRR_{\Omega \setminus \Nb_\ell(\omega)}^2 
	\leq \delta^{2\ell} \LLL u_h^{3/2}\RRR^2
	$
	and particularly~\eqref{eq:dec} for $n = 1$ with $\delta \leq \gamma < 1$.%\medskip
	
	Let now $n \geq 2$ and assume that~\eqref{eq:dec} holds for any $k < n$. 
	We abbreviate $w:=I_h(u_h^{n+1/2} \eta_\ell)$ and observe from
	\eqref{eq:H1cutoff} that $\|w\|_a$
	is bounded from above by a constant times
	$\|u_h^{n+1/2}\|_a$.
	For the last term in~\eqref{eq:proofDec4}, we therefore estimate
	\begin{align*}
		&\mathcal K(u_h^{n+1/2},I_h(u_h^{n+1/2} \eta_\ell))
		=\mathcal K(u_h^{n+1/2},w)
		\\
		& 
		\quad= \tau^2 (\widehat{f_h^{n-1/2}},w) 
		+ (2 u_h^{n-1/2} - u_h^{n-3/2},w)
		- \tfrac{\tau^2}{4}a(2u_h^{n-1/2} + u_h^{n-3/2}, w)\\
		&
		\quad\leq \tfrac12 \LLL u_h^{n+1/2}\RRR^2_{\Omega\setminus \Nb_{\ell-1}(\omega)} 
		+ C_2\big(\tfrac12\LLL u_h^{n-1/2}\RRR^2_{\Omega\setminus \Nb_{\ell-1}(\omega)} + \tfrac12\LLL u_h^{n-3/2}\RRR^2_{\Omega\setminus \Nb_{\ell-1}(\omega)}\big)
	\end{align*}
	with an appropriate constant 
	$C_2$ (proportional to $\beta\alpha^{-1}$)
	using a weighted Young's inequality. 
	Going back to~\eqref{eq:proofDec4} (with the specific choice $z_h = u_h^{n+1/2}$), we therefore get 
	\begin{align*} 
		\LLL u_h^{n+1/2}\RRR_{\Omega \setminus \Nb_\ell(\omega)}^2
		&
		\hspace{-0.08cm}\leq\hspace{-0.05cm}  C_{\tau,h} \LLL u_h^{n+1/2}\RRR_{R_\ell}^2\hspace{-0.07cm} 
		+ \hspace{-0.07cm}\tfrac12 \LLL u_h^{n+1/2}\RRR_{\Omega\setminus \Nb_{\ell-1}(\omega)}^2 
		\hspace{-0.08cm}+\hspace{-0.07cm}\tfrac{C_{\tau,h}}{2}\hspace{-.2cm} \sum_{\nu=1,3}\hspace{-0.05cm}\LLL u_h^{n-\nu/2}\RRR^2_{\Omega\setminus \Nb_{\ell-1}(\omega)},
	\end{align*}
	where $C_{\tau,h} := \max\{C_1,C_2\} \eqsim \beta\alpha^{-1}(\tau/h+h/\tau)$ with $C_1$ as defined after~\eqref{eq:proofDec4}. Rearranging terms, we arrive at
	\begin{equation*} 
		(1\hspace{-0.05cm}+\hspace{-0.05cm} C_{\tau,h})\LLL u_h^{n+1/2}\RRR_{\Omega \setminus \Nb_\ell(\omega)}^2
		\hspace{-0.05cm}\leq\hspace{-0.05cm} (C_{\tau,h} \hspace{-0.05cm}+\hspace{-0.05cm} \tfrac12) \LLL u_h^{n+1/2}\RRR_{\Omega\setminus \Nb_{\ell-1}(\omega)}^2 
		\hspace{-0.05cm}+\hspace{-0.05cm}\tfrac{C_{\tau,h}}{2} \hspace*{-0.2cm}\sum_{\nu=1,3}\hspace{-0.05cm}\LLL u_h^{n-\nu/2}\RRR^2_{\Omega\setminus \Nb_{\ell-1}(\omega)}.
	\end{equation*}
	Dividing by $(1 + C_{\tau,h})$ on both sides and setting 
	$\gamma := \sqrt{(C_{\tau,h} + \tfrac12)/(1+C_{\tau,h})} < 1$ (note that $\delta \leq \gamma$), we further obtain
	\begin{equation*}
		\begin{aligned}
			\LLL u_h^{n+1/2}\RRR_{\Omega \setminus \Nb_\ell(\omega)}^2 
			\leq \gamma^2\big(\LLL u_h^{n+1/2}\RRR_{\Omega \setminus \Nb_{\ell-1}(\omega)}^2 
			+ \tfrac12\sum_{\nu=1,3}\LLL u_h^{n-\nu/2}\RRR^2_{\Omega\setminus \Nb_{\ell-1}(\omega)} 
			\big)
		\end{aligned}
	\end{equation*}
	and iterating the argument for $j=\ell-1,\dots,1$ we get
	\begin{align*}
		\LLL u_h^{n+1/2}\RRR_{\Omega \setminus \Nb_\ell(\omega)}^2 
		&\leq \gamma^{2\ell}\LLL u_h^{n+1/2}\RRR^2
		+ \tfrac12\sum_{j=0}^{\ell-1}\gamma^{2(\ell-j)}
		\sum_{\nu=1,3}\LLL u_h^{n-\nu/2}\RRR^2_{\Omega\setminus \Nb_{j}(\omega)}.
	\end{align*}
	From the assumption that~\eqref{eq:dec} holds for any $k < n$
	we obtain for the sum on the right-hand side 
	and $M_{u_h}:=\max_{k \leq n}\LLL u_h^{k+1/2}\RRR$
	\begin{align*}
		\tfrac12\sum_{j=0}^{\ell-1}\gamma^{2(\ell-j)}
		\sum_{\nu=1,3}\LLL u_h^{n-\nu/2}\RRR^2_{\Omega\setminus \Nb_{j}(\omega)}
		&\leq 
		\sum_{j=0}^{\ell-1}\gamma^{2(\ell-j)}(j+1)^{(n-1)}\gamma^{2j} M_{u_h}^2
		\\&
		\leq \ell(\ell+1)^{(n-1)}\gamma^{2\ell} M_{u_h}^2,
	\end{align*}
	so that altogether
	\begin{align*}
		\LLL u_h^{n+1/2}\RRR_{\Omega \setminus \Nb_\ell(\omega)}^2 
		\leq 
		(\ell+1)^n\gamma^{2\ell}
		\max_{k \leq n}\LLL u_h^{k+1/2}\RRR^2,
	\end{align*}
	which is the bound asserted in \eqref{eq:dec}.
\end{proof}

\subsection{Localized computations}

We now turn to the computation of a discrete solution to the wave equation on localized patches. Therefore, assume that the initial conditions of~\eqref{eq:CN} fulfill $u_h^1,\,u_h^{0} \in V_h(\omega)$ and $f$ has support in $\omega \subseteq \Omega$. The localized variant of~\eqref{eq:CN} now seeks $(\tilde u_h^{n})_{n \geq 2} \in V_h(\Nb_\ell(\omega))$ such that 
\begin{equation}\label{eq:CNloc}
	(\hat\partial^2_\tau\tilde u_h^n,v_h)_{\Nb_\ell(\omega)}
	+ a_{\Nb_\ell(\omega)}(\widehat{\tilde u_h^n},v_h)
	= (\widehat{f_h^{n}},v_h)_{\Nb_\ell(\omega)}
	\quad\text{for all }v_h \in V_h(\Nb_\ell(\omega))
\end{equation}
for some $\ell \in \N$ given the initial conditions $\tilde u_h^1=u_h^1,\,\tilde u_h^0=u_h^{0} \in V_h(\omega)$. 
We define $\tilde u_h^{n+1/2}$ following the convention \eqref{eq:halfstep}.
We emphasize that~\eqref{eq:CNloc} includes zero boundary conditions for $\tilde u_h^n$ on the boundary of $\Nb_\ell(\omega)$ due to the definition of the space~$V_h(\Nb_\ell(\omega))$.
Next, we use Lemma~\ref{lem:decay} in order to show that the discrete solutions to the wave equation with local data can be computed in a localized way with only a reasonable impact on the overall error. 

\begin{theorem}[Localization error]\label{thm:locerr}
	Let the assumptions of Lemma~\ref{lem:decay} hold. 
	Then 
	for $\ell \in \N$
	the error 
	$\xi^n:=\tilde u_h^{n+1/2} - u_h^{n+1/2}$
	between the solutions $(u_h^{n})_{n \geq 0}$ of~\eqref{eq:CN} 
	and $(\tilde u_h^n)_{n \geq 0}$ of~\eqref{eq:CNloc} fulfills
	\begin{equation*}
		\LLL \xi^n \RRR
		=\LLL u_h^{n+1/2} - \tilde u_h^{n+1/2}\RRR
		\lesssim C_{\tau,h} (9\ell)^{n/2}\gamma^{\ell}\max_{k \leq n} \LLL u_h^{k+1/2}\RRR
	\end{equation*}
	for any $n \in \N_0$ and $\gamma$ and $C_{\tau,h}$ as in Lemma~\ref{lem:decay}.
\end{theorem}

\begin{proof}
	For $n=0,1$ the result follows from $\xi^n=0$ 
	by the choice of initial conditions. 
	For $n\geq1$, we define
	$z_h^{n+1} \in V_h(\Nb_{\ell}(\omega))$ as the solution to
	\begin{equation*}
		\tau^{-2}(z_h^{n+1} - 2 u_h^n + u_h^{n-1},v_h)_{\Nb_\ell(\omega)} 
		+ \tfrac14 a_{\Nb_\ell(\omega)}( z_h^{n+1} + 2 u_h^n + u_h^{n-1}, v_h)
		= (\widehat{f_h^n},v_h)_{\Nb_\ell(\omega)}
	\end{equation*}
	for all $v_h \in V_h(\Nb_\ell(\omega))$.
	We define
	$z_h^{n+1/2}$ following \eqref{eq:halfstep} and observe that
	$\mu^n:=\tilde u_h^{n+1/2} - z_h^{n+1/2}$ solves
	\begin{equation*}
		\mathcal K_{\Nb_\ell(\omega)}(\mu^n,v_h) 
		= (2 \xi^{n-1} - \xi^{n-2},v_h)_{\Nb_\ell(\omega)}
		- \tfrac{\tau^2}4a_{\Nb_\ell(\omega)}(2\xi^{n-1} + \xi^{n-2}, v_h)
	\end{equation*}
	for $n \geq 2$. 
	With the choice $v_h = \mu^n$, this leads to 
	\begin{equation*}
		\LLL\mu^n\RRR
		\leq 2 \LLL \xi^{n-1}\RRR + \LLL\xi^{n-2}\RRR .
	\end{equation*}
	Using this estimate and the triangle inequality repeatedly, 
	we obtain with $\nu^{n}:=u_h^{n+1/2} - z_h^{n+1/2}$ that
	\begin{equation}\label{eq:proofLoc1}
		\LLL \xi^{n} \RRR
		\leq \LLL \nu^{n} \RRR 
		+ \LLL \mu^n \RRR
		\leq \LLL \nu^{n} \RRR
		+ {2} \LLL \xi^{n-1} \RRR
		+ \LLL \xi^{n-2} \RRR
		\leq \sum_{k=0}^{n} 3^{n-k} \LLL \nu^{k} \RRR
	\end{equation}
	since the initial conditions for~\eqref{eq:CN} and~\eqref{eq:CNloc} coincide.
	Due to 
	\begin{equation*} 
		\mathcal K(\nu^{k},v_h)
		=
		\mathcal K(u_h^{k+1/2} - z_h^{k+1/2},v_h)
		= 0
		\qquad\text{for all }
		v_h \in V_h(\Nb_\ell(\omega)), 
	\end{equation*}
	we obtain the best-approximation property
	\begin{equation*}
		\LLL \nu^{k}\RRR
		\leq  \LLL u_h^{k+1/2} - w\RRR
		\quad\text{for any }w\in V_h(\Nb_\ell(\omega)).
	\end{equation*}
	With the cutoff function $\eta_\ell$ defined in~\eqref{eq:eta} and the nodal interpolation operator~$I_h$, we 
	define the particular choice $w:=I_h((1-\eta_\ell)u_h^{k+1/2})$.
	With the arguments from the proof of Lemma~\ref{lem:decay},
	particularly~\eqref{eq:stabIh} and~\eqref{eq:proofDec3}, 
	we get with 
	$R_\ell =  \Nb_{\ell}(\omega) \setminus \Nb_{\ell-1}(\omega)$ that
	\begin{equation*}
		\begin{aligned}    
			\LLL \nu^{k}\RRR
			&\leq \LLL u_h^{k+1/2} - w\RRR
			= \LLL u_h^{k+1/2} - w\RRR_{\Omega\setminus\Nb_{\ell-1}(\omega)} 
			\\
			&
			\leq \LLL u_h^{k+1/2}\eta_\ell\RRR_{\Omega\setminus\Nb_{\ell-1}(\omega)} 
			+ \LLL (\id - I_h)((1-\eta_\ell)u_h^{k+1/2})\RRR_{\Omega\setminus\Nb_{\ell-1}(\omega)}
			\\
			&
			\lesssim  C_{\tau,h}\,\LLL u_h^{k+1/2}\RRR_{\Omega\setminus\Nb_{\ell-1}(\omega)}
		\end{aligned}
	\end{equation*}
	with $C_{\tau,h}$ from Lemma~\ref{lem:decay}.  
	Inserting this bound in
	\eqref{eq:proofLoc1} and using Lemma~\ref{lem:decay}, we further get 
	\begin{equation*}
		\begin{aligned}
			\LLL \xi^n\RRR
			&\lesssim C_{\tau,h}\sum_{k=0}^{n}3^{n-k}\ell^{k/2}\gamma^{\ell-1} \max_{j \leq k}\LLL u_h^{j+1/2}\RRR
			\lesssim 
			C_{\tau,h}\,3^{n+1}\ell^{n/2}\gamma^{\ell-1} \max_{k \leq n}\LLL u_h^{k+1/2}\RRR.
		\end{aligned}
	\end{equation*}
	Therefore, the asserted estimate follows.  
\end{proof}

\begin{remark}[Choice of $\ell$]\label{rem:ell}
	Aiming for an error of order $\mathcal O(\varepsilon)$ in Theorem~\ref{thm:locerr} for a fixed $n \in \N$, we may set
	\begin{equation*}
		C_{\tau,h}(9\ell)^{n/2} \gamma^\ell = \varepsilon,
	\end{equation*}
	from which we get that
	\begin{equation*}
		\log_\gamma(C_{\tau,h}) + \tfrac{n}{2}\log_\gamma(9\ell) + \ell = \log_\gamma(\varepsilon).
	\end{equation*}
	This results in 
	\begin{equation*}
		\ell \eqsim \tfrac{1}{|\log \gamma|} \big( n \log \ell + |\log \varepsilon| + \log C_{\tau,h} \big)
		.
	\end{equation*}
	Using that 
	\begin{equation*}
		|\log \gamma| = \tfrac12|\log\big((C_{\tau,h} + \tfrac12)/(1+C_{\tau,h})\big)| = \tfrac12 \log(1+\tfrac{1}{1 + 2 C_{\tau,h}})
	\end{equation*}
	and $\log(1+\tfrac1x) = - \sum_{n=1}^\infty n^{-1}(-x)^{-n} \leq \tfrac1x$, which is a good approximation for $x \gg 1$, we deduce the essential scaling
	\begin{equation*}
		\frac{1}{\lvert \log \gamma\rvert} \eqsim C_{\tau,h} 
	\end{equation*}
	and thus
	\begin{equation*}
		\ell \eqsim C_{\tau,h} n \log \ell + C_{\tau,h}\lvert\log \varepsilon\rvert + C_{\tau,h} \log C_{\tau,h}.
	\end{equation*}
	Note that $\ell \gtrsim C_{\tau,h} n \log \ell$ is satisfied for 
	$\ell \gtrsim -(C_{\tau,h} n)\,W_{-1}(-\tfrac{1}{C_{\tau,h} n})$, 
	where $W_{-1}$ denotes a branch of the usual Lambert $W$ function (product logarithm).
	For $0<x\leq \exp(-1)$ and $y:=W_{-1}(-x)$, the functional identity
	$-x=y\exp(y)$ and elementary estimates yield 
	\begin{equation*} 
		1-\exp(-1) \leq \frac{\log(1/x)}{-W_{-1}(-x)} \leq 1,
	\end{equation*}
	therefore $\ell\gtrsim (C_{\tau,h} n) \log(C_{\tau,h} n)$
	implies $\ell \gtrsim -(C_{\tau,h} n)\,W_{-1}(-\tfrac{1}{C_{\tau,h} n})$
	and furthermore $\ell \gtrsim C_{\tau,h} \log C_{\tau,h}$.
	This gives the sufficient condition
	\begin{equation}\label{eq:scalingEll}
		\ell 
		\gtrsim
		C_{\tau,h} \Big(n\,\big(\log n + \log C_{\tau,h}\big)  
		+ \lvert\log\varepsilon\rvert\Big)
	\end{equation}
	for the scaling of $\ell$. 
	Thus, if~\eqref{eq:scalingEll} is satisfied, 
	there exists a constant $C > 0$, which is independent of $n$, $\varepsilon$, and $\ell$, 
	but depends on $d$, $\alpha$, and $\beta$, such that
	the estimate of Theorem~\ref{thm:locerr} simplifies to 
	\begin{align*}
		\LLL u_h^{n+1/2} - \tilde u_h^{n+1/2} \RRR
		\lesssim C_{\tau,h} (9\ell)^{n/2}\gamma^{\ell}\max_{k \leq n}\LLL u_h^{k+1/2}\RRR
		\leq C\, \varepsilon \max_{k \leq n}\LLL u_h^{k+1/2}\RRR.
	\end{align*}
	That is, to retain a reasonable approximation after $n$ time steps when computing the discrete solution to~\eqref{eq:CN} locally (given that $f$, $u_h^1$, and $u_h^0$ have support in a local domain~$\omega$), 
	we essentially need to extend $\omega$ by 
	$\mathcal O(\tau n/h \log (\tau n/h))$ layers of fine elements
	in the relevant regime $\tau\geq h$.
	We emphasize that in order to capture the physically expected wave cone,
	an extension of the support $\omega$ by at least 
	$\mathcal O(\tau n/h)$ additional layers is reasonable,		
	which corresponds to the number of layers that a locally defined wave travels within the time frame $\tau n$ (provided that the wave speed is of order $\mathcal O(1))$. 
	In that regard, our result only requires a logarithmic overhead to incorporate the fact that the discrete wave is not fully local. We also refer to Section~\ref{ss:choiceparams} for a discussion of the choice of parameters. 
\end{remark}

\section{Local superposition method}\label{s:LSM}

In this section, we utilize the localization result of the previous section to construct a fairly simple local superposition strategy that is easily parallelizable. To obtain local data, initial conditions and the right-hand side are localized by a partition of unity, where we exploit the linearity of the wave equation. The approach may be understood as a domain decomposition strategy in space on successive coarse time intervals that does neither require multiple iterations nor a sophisticated definition of boundary conditions between the different sub-regions.

\subsection{The algorithm}
Before we state the full algorithm, we require some additional definitions. 
Let $\calT_H$ be a regular and quasi-uniform mesh with mesh size $H$ and assume for simplicity that the mesh $\calT_h$ is a refinement of $\calT_H$ and $H/h \in \N$.
The assumption of nestedness is not necessary but simplifies the overall presentation. 
We also introduce the coarse time step size $T$ with $T/\tau \in \N$ to be specified later.
Let $\Lambda_i,\,i=1,\ldots,M$ be the nodal $Q_1$ basis functions corresponding to the mesh $\calT_H$. In particular,
\begin{equation}\label{eq:Lambda_properties}
	\sum_{i=1}^M \Lambda_i \equiv 1,
	\quad
	0 \leq \Lambda_i \leq 1,\,\;i = 1,\ldots,M,
	\quad
	\|\nabla \Lambda_i\|_{L^\infty(\Omega)} \lesssim H^{-1} .
\end{equation}
We denote with $\omega_i = \supp \Lambda_i$ the support of $\Lambda_i$ that consists of $2^d$ (coarse) elements. 

The idea of our approach is to approximate the solution $(u_h^n)_n$ of~\eqref{eq:CN} by localized computations only. More precisely, let $u_h^1,\,u_h^0 \in V_h$ and $f_h^n \in V_h$, $n \in \N_0$ be the (possibly globally supported) discrete initial conditions and right-hand side functions of~\eqref{eq:CN}. 
For a given~$i$, let 
\begin{equation}\label{eq:dataLoc}
	f_{h,i}^n := I_h(\Lambda_if_h^n), \quad u_{h,i}^1 = I_h(\Lambda_iu_h^1), \quad u_{h,i}^0 = I_h(\Lambda_iu_h^0). 
\end{equation} 
These initial conditions and right-hand side functions are only locally supported on the domain $\omega_i$. Therefore, we can compute the respective solutions locally as in~\eqref{eq:CNloc}. More precisely, we compute the local solutions
\begin{equation}\label{eq:CNloci}
	(\hat\partial^2_\tau\bar u_{h,i}^n, v_h)_{\Nb_\ell(\omega_i)}
	+ a_{\Nb_\ell(\omega_i)}(\widehat{\bar{u}_{h,i}^n},  v_h)
	= (\widehat{f_{h,i}^{n}}, v_h)_{\Nb_\ell(\omega_i)}
	\quad\text{for all } v_h \in V_h(\Nb_\ell(\omega_i)),
\end{equation}  
for $n \leq T/\tau$, and with
\begin{equation}\label{eq:choiceEll}
	\ell \gtrsim C_{\tau,h} \Big(\tfrac{T}{\tau} 
	\log \big(C_{\tau,h}\tfrac{T}{\tau}\big) + |\log h| 
	+ |\log \vartheta| + \tfrac{t_\mathrm{fin}}{T}|\log H|\Big),
\end{equation}
where $C_{\tau,h} \eqsim \max\{\tau/h,h/\tau\}$ as before. 
The stated assumption on $\ell$ will become clear in Theorem~\ref{thm:superpositionError} and will be further discussed in Section~\ref{ss:choiceparams}. Note that the hidden constant in~\eqref{eq:choiceEll} depends on $\alpha$ and $\beta$. However, these dependencies are not severe as observed in the numerical examples.
Recall that the local problems~\eqref{eq:CNloci} are solved with zero boundary conditions. In particular, specifically designed conditions to connect different patches are not required.
The aim of our local superposition strategy is to achieve an error of order $\mathcal{O}(\vartheta)$, where $\vartheta$ 
optimally scales like the error of the global Crank--Nicolson scheme, i.e., $\vartheta \eqsim h + \tau^2$. This is due to the fact that the distance to the exact solution cannot scale better than the error of the classical Crank--Nicolson solution. 
Given the respective initial conditions, these local solutions are completely independent of each other, such that the computations can be parallelized. If we are interested in the global solution at a certain time step, we can sum up the corresponding local contributions.

After $\nr := T/\tau$ time steps, we reset the algorithm to avoid the computation on patches that require a too large choice of $\ell$. More precisely, the summed up discrete solutions 
\begin{equation*}
	\bar u_h^{\nr+1} := \sum_{i=1}^M u_{h,i}^{\nr+1}\qquad\text{and}\qquad\bar u_h^{\nr} := \sum_{i=1}^M u_{h,i}^{\nr}
\end{equation*}
can be used as new initial conditions, which are once again decomposed into local contribution through the partition of unity $\{\Lambda_i\}_{i=1}^M$. This strategy is repeated after another $\nr$ time steps and ensures that the computational domains do not grow too much. The usage of the notation $\bar{(\bullet)}$ instead of $\tilde{(\bullet)}$ refers to these local solutions that are regularly reset.

\begin{algorithm}[t]
	\caption{The local superposition method.}\label{algo:wave}
	\begin{algorithmic}
		\State % dummy
		\Require $f$, $u_0$, $v_0$, $H$, $h$, $T$, $\tau$, $t_\mathrm{fin}$, $\ell$
		\State $N_T:=t_\mathrm{fin}/T$
		\State Compute initial conditions $\mathfrak{a}$ at time $0$ and $\mathfrak{b}$ at time $\tau$ using $f$, $u_0$, $v_0$
		\For{$k=1,\dots,N_T$}
		\For{$i=1,\dots,M$}
		\State set $f_i^{(k)}=\Lambda_i f(\bullet,k\,T+\bullet)$
		\State
		$(\mathfrak{a}_i,\mathfrak{b}_i)
		=
		\mathsf{CN}(\Nb_{\ell}(\omega_i),T,h,\tau,f_i^{(k)},\Lambda_i\mathfrak{a},\Lambda_i\mathfrak{b})$
		\EndFor
		\State $\mathfrak{b} := \sum_{i=1}^M \mathfrak{b}_i$ 
		\State $\mathfrak{a} :=  \sum_{i=1}^M \mathfrak{a}_i$ 
		\State $\bar u^{kT}_h := \mathfrak{a}$
		\EndFor
		\Ensure $\bar u_h^{T}$, $\bar u_h^{2T}$, \ldots
		\State
		\Function{$\mathsf{CN}$}{$\omega$, $t$, $h$, $\tau$, $g$, ${a}$, ${b}$}
		\State Solve Crank--Nicolson over the domain $\omega$ with time horizon 
		$t + \tau$, mesh size~$h$, 
		\State step size $\tau$,
		right-hand side $g$, initial data ${a}$, ${b}$; return the solution at the last 
		\State two time steps (i.e., at $t$ and $t+\tau$)
		
		\EndFunction
	\end{algorithmic}
\end{algorithm}

For an illustration of the method, we refer to Algorithm~\ref{algo:wave}. It requires the given data, the coarse and fine mesh sizes and time steps $H$, $h$, $T$, and $\tau$, as well as the patch size $\ell$. The main ingredient is a standard (localized) Crank--Nicolson method on different patches combined with the successive localization of the data. The algorithm is written in a way that only the global solutions $\bar u_h^T$, $\bar u_h^{2T}$, etc.\ at the coarse time points $kT$ are stored, where $k$ indicates the number of resets. If the solution at a specific (fine) time point between these coarse time points is required, the respective local solutions could be stored and summed up as well. Alternatively, they can be cheaply computed retroactively starting from the stored solution at the previous coarse time point.

\subsection{Error analysis}

To bound the error between the solution $\bar u_h^n$ and the global Crank--Nicolson solution~$u_h^n$ at the time point $n\tau$, we 
now state and prove an error estimate.
The proof makes use of the discrete norm
\begin{equation*}
	\|v_h\|_{{\mathcal E}_h}:=\big(\LLL\Dt  v_h\RRR^2
	+ \tau^{-2}\LLL v_h\RRR^2 \big)^{1/2}
\end{equation*}
for any $v_h \in V_h$ (that we implicitly assume to be defined for successive time steps). 
We obtain as a consequence of the upper bound on $A$
and an inverse estimate that
\begin{equation*}
	\|\Dt v_h\|^2 \leq
	\LLL\Dt v_h\RRR^2
	\lesssim (1+\tau^2h^{-2}) \|\Dt v_h\|^2.
\end{equation*}
Furthermore, the Poincar\'e-Friedrichs estimate
and the lower bound on $A$ show
\begin{equation*}
	\|v_h\|_a^2
	\leq
	4\tau^{-2} \LLL v_h \RRR^2
	\lesssim
	\tau^{-2} \|v_h\|_a^2.
\end{equation*}
From these norm equivalences
we deduce that the discrete norm 
is equivalent to the 
energy norm in the sense that 
\begin{equation}\label{eq:equivnorm}
	\min\{\tau,h/\tau\}\, \| v_h\|_{{\mathcal E}_h} 
	\lesssim
	\| v_h\|_{{\mathcal E}} \lesssim \| v_h\|_{{\mathcal E}_h}.
	%\quad\text{for any }v_h \in V_h.
\end{equation}
The (hidden) constants depend on the bounds on $A$.
Apart from these norm equivalences, we also make use of the following observation: for any $n \in \N_0$, we have for the non-localized Crank--Nicolson solution~$u_h^n$ (cf.~\eqref{eq:CN}) that
$
u_h^n = \sum_{i=1}^M u_{h,i}^n,
$
where $u_{h,i}^n$, $m+1 < n \leq (k+1)\nr + 1$ with $m = k\,\nr$ for some $k \in \N_0$ is the solution to
\begin{equation}\label{eq:CNDDi}
	(\hat\partial^2_\tau u_{h,i}^n,v_h)_{\Nb_\ell(\omega_i)} 
	+
	a_{\Nb_\ell(\omega_i)}(\widehat{u_{h,i}^n}, v_h)
	= (\widehat{f_{h,i}^{n}},v_h)_{\Nb_\ell(\omega_i)}
	\quad\text{for all }v_h \in V_h(\Nb_\ell(\omega_i))
\end{equation}
with initial conditions $u_{h,i}^{m+1} := I_h(\Lambda_iu_h^{m+1})$ and $u_{h,i}^m := I_h(\Lambda_iu_h^m)$.
This follows from the linearity of the scheme~\eqref{eq:CN} (with respect to the right-hand side and the initial conditions) and the observation that by construction of the localized initial data and right-hand side in~\eqref{eq:dataLoc}, it holds that
\begin{equation*}
	f_h^n = \sum_{i=1}^Mf_{h,i}^n,\qquad u_h^{m+1} = \sum_{i=1}^Mu_{h,i}^{m+1}, \qquad u_h^{m} = \sum_{i=1}^Mu_{h,i}^m.
\end{equation*}
We now have all tools at hand to prove an error estimate for the local superposition method.

\begin{theorem}[Error of the local superposition method]\label{thm:superpositionError}
	Let $\ell$ be chosen as in~\eqref{eq:choiceEll} and $\vartheta > 0$. 
	Then for any $n \in \N_0$, $n \leq t_\mathrm{fin}/\tau =:N$, we have the error estimate
	\begin{equation*}
		\| \bar u_h^{n+1/2} - u_h^{n+1/2} \|_{\mathcal E}
		\lesssim
		\vartheta \Big(\|u_{h}^{1/2}\|_{\mathcal E}
		+
		\sum_{j=1}^{N} \tau \big\|\widehat{f_h^{j}}\big\| \Big),
	\end{equation*}
	where the hidden constant depends on 
	$\alpha,\,\beta$, and $d$. 
\end{theorem}

\begin{proof}[Proof of Theorem~\ref{thm:superpositionError}]
	Let a time step $n \in \N$ be fixed and let $m < n -1$ be the point in time where the domain decomposition strategy was re-started for the last time (or $m=0$ for $n \leq \nr + 1$). We denote with $\tilde z_h^n$, $n \geq m$ the auxiliary localized solution computed as in~\eqref{eq:CNloci} from the starting values
	\begin{equation*}
		\tilde z_{h,i}^{m+1} := I_h(\Lambda_iu_h^{m+1}),\qquad \tilde z_{h,i}^{m} := I_h(\Lambda_iu_h^{m})
		\qquad 
		\text{for } i = 1,\ldots,M.
	\end{equation*}
	By Theorem~\ref{thm:locerr} and Remark~\ref{rem:ell}, we can 
	guarantee with the choice of~$\ell$ in~\eqref{eq:choiceEll} that 
	\begin{equation}\label{eq:erri}
		\LLL u_{h,i}^{n+1/2} - \tilde z_{h,i}^{n+1/2} \RRR 
		\lesssim \vartheta \,
		\min\{\tau,h/\tau\}
		\,H^{\kappa+2} \max_{m \leq k \leq m+\nr}\LLL u_{h,i}^{k+1/2}\RRR
	\end{equation}
	for $m \leq n \leq m+\nr$, 
	where we abbreviate
	$\kappa:=t_\mathrm{fin}/T$
	and explicitly choose $\varepsilon \eqsim \vartheta\min\{\tau,h/\tau\}\, H^{\kappa + 2}$. 
	Note that the result in Theorem~\ref{thm:locerr} holds for the discrete temporal derivative as well, employing linearity arguments. That is, in our setting we also have
	\begin{equation}\label{eq:erri2}
		\LLL \Dt u_{h,i}^{n+1/2} - \Dt\tilde z_{h,i}^{n+1/2}\RRR
		\lesssim \vartheta\,
		\min\{\tau,h/\tau\}
		\,H^{\kappa+2}\max_{m \leq k \leq m+\nr} \LLL \Dt u_{h,i}^{k+1/2}\RRR.
	\end{equation}

	We now estimate the global error. 
	We define $\bar e^{n}:=\bar u_h^{n} - u_h^{n}$
	as well as
	$e_z^{n}:=\tilde z_h^{n} - u_h^{n}$
	and $\bar e_z^{n}:=\tilde z_h^{n} - \bar u_h^{n}$.
	By the triangle inequality, we have
	\begin{equation}\label{eq:proofGlobal}
		\| \bar e^{n+1/2} \|_{\mathcal E}
		\leq 
		\| \bar e_z^{n+1/2} \|_{\mathcal E} 
		+ \|e_z^{n+1/2}\|_{\mathcal E}
		.
	\end{equation}
	The last term on the right-hand side can be estimated using~\eqref{eq:equivnorm}
	and the decomposition into local subproblems.
	Indeed, abbreviating $e_{z,i}^n= \tilde z_{h,i}^{n} - u_{h,i}^{n}$,
	we have from~\eqref{eq:equivnorm}
	\begin{equation*}
		\| e_z^{n+1/2} \|_{\mathcal E}
		\lesssim \| e_z^{n+1/2}\|_{\mathcal E_h}
		\lesssim
		\Big(\sum_{i=1}^M  \| e_{z,i}^{n+1/2}\|_{\mathcal E_h}^2\Big)^{1/2}
	\end{equation*}
	and deduce from \eqref{eq:erri} and \eqref{eq:erri2} that
	\begin{equation}\label{eq:gloloc}
		\| e_z^{n+1/2} \|_{\mathcal E}
		\lesssim 
		\vartheta H^{\kappa+2}
		\Big(
		\sum_{i=1}^M 
		\max_{m < k \leq m+\nr}
		\min\{\tau^2,(h/\tau)^2\}
		\| u_{h,i}^{k+1/2}\|_{\mathcal E_h}^2
		\Big)^{1/2}.
	\end{equation}
	In what follows we use the shorthand notation
	\begin{equation*}
		\mathcal F_{r,i}^k:=\sum_{j=r}^{k} \tau \big\|\Lambda_i\widehat{f_h^{j}}\big\|
		\quad\text{and}\quad
		\mathcal F_r^k:=\sum_{j=r}^{k} \tau \big\|\widehat{f_h^{j}}\big\|.
	\end{equation*}
	With~\eqref{eq:equivnorm} and the stability estimate in Theorem~\ref{thm:energy},
	we get
	\begin{align*}
		\min\{\tau,h/\tau\}
		\| u_{h,i}^{k+1/2}\|_{\mathcal E_h}
		& \lesssim \|u_{h,i}^{m+1/2}\|_{\mathcal E , \omega_i} 
		+  \mathcal F_{m,i}^{m+\nr}
		\\
		&
		\lesssim (1+H^{-1}) \| u_{h}^{m+1/2}\|_{\mathcal E,\omega_i}
		+  \mathcal F_{m,i}^{m+\nr},
	\end{align*}
	where we use the definition of $u_{h,i}^{m+1/2}$, the product rule,
	and the bound \eqref{eq:Lambda_properties} in the last step.
	Going back to~\eqref{eq:gloloc} and leveraging the limited 
	overlap of the supports~$\omega_i$, we obtain
	\begin{equation}\label{eq:gloloc2}
		\| e_z^{n+1/2} \|_{\mathcal E}
		\lesssim \vartheta H^{\kappa+1}\Big[\| u_{h}^{m+1/2}\|_{\mathcal E} 
		+  \mathcal F_m^{m+\nr}\Big] 
		\lesssim \vartheta H^{\kappa+1}\Big[\|u_{h}^{1/2}\|_{\mathcal E} 
		+  \mathcal F_1^N\Big].
	\end{equation}
	To bound the first term on the right-hand side of~\eqref{eq:proofGlobal}, 
	we observe that the functions
	$\bar e_{z,i}^n$,
	$n > m+1$ solve the equation 
	\begin{equation*}
		(\hat\partial^2_\tau \bar e_{z,i}^n,v_h)_{\Nb_\ell(\omega_i)}
		+ a_{\Nb_\ell(\omega_i)}(\widehat{\bar e_{z,i}^n}, v_h)
		= 0
		\quad\text{for all }v_h \in V_h(\Nb_\ell(\omega_i)).
	\end{equation*}
	With Theorem~\ref{thm:energy}, we obtain the error estimate 
	\begin{align*}
		\| \bar e_{z,i}^{n+1/2}\|_{\mathcal E,\Nb_\ell(\omega_i)}
		\lesssim \|\bar e_{z,i}^{m+1/2}\|_{\mathcal E,\omega_i}
		\lesssim \|\Dt \bar e^{m+1/2}\|_{\omega_i} + (1+H^{-1})\|\bar e^{m+1/2}\|_{a,\omega_i},
	\end{align*}
	where we have used the initial condition
	$\bar e_{z,i}^{m+1/2}= -I_h(\Lambda_i \bar e^{m+1/2})$,
	the stability~\eqref{eq:H1cutoff}, and the
	product rule with the bound from \eqref{eq:Lambda_properties}.
	In particular, we have
	\begin{equation}\label{eq:gloloc3}
		\|\bar e_z^{n+1/2}\|_{\mathcal E}
		\lesssim \|\Dt \bar e^{m+1/2}\| + H^{-1}\|\bar e^{m+1/2}\|_a.
	\end{equation}
	Finally, going back to~\eqref{eq:proofGlobal} and using~\eqref{eq:gloloc2} and~\eqref{eq:gloloc3} multiple times, we obtain
	\begin{align*}
		\|\bar e^{n+1/2}\|_{\mathcal E} 
		&
		\lesssim \sum_{j=1}^{N/\nr} \vartheta H^{\kappa+1} H^{-j}\Big[\| u_{h}^{1/2}\|_{\mathcal E}  +  \mathcal F_1^N \Big]
		\\
		& \lesssim \vartheta  H^{\kappa+1} H^{-N/\nr-1}\Big[\| u_{h}^{1/2}\|_{\mathcal E}  +  \mathcal F_1^N\Big] 
		\lesssim \vartheta \Big[\| u_{h}^{1/2}\|_{\mathcal E}  +  \mathcal F_1^N\Big],
	\end{align*}
	where we use that by definition $\nr = T/\tau$ and
	thus $N/\nr = t_\mathrm{fin}/T=\kappa$.
	This is the assertion. 
\end{proof}
\begin{remark}[Generalizations] Instead of~\eqref{eq:CN}, we could as well employ the more general \emph{Newmark scheme} \cite{New59}
	\begin{equation}\label{eq:Newmark}
		(\hat\partial^2_\tau u_h^n,v_h)_{L^2(\Omega)} 
		+
		a(\widehat{u_h^n},v_h) = (\widetilde{f_h^{n}},v_h)_{L^2(\Omega)}
		\quad\text{for all }v_h \in V_h,
	\end{equation}
	with the alternative definition 
	\begin{equation*}
		\widehat{u_h^n} := \tfrac12\big( 2 c_1 u_h^{n+1} + (1-4c_1 + 2 c_2)u_h^n + (1 + 2c_1 - 2c_2)u_h^{n-1} \big)
	\end{equation*}
	for given parameters $c_1,\,c_2$.
	Here, $\widetilde{f_h^{n}}$ realizes an appropriate weighting of $f_h^{n+1}$, $f_h^{n}$, and $f_h^{n-1}$. Our results directly carry over to this scheme up to a change in constants. We emphasize that the choice $c_1 = 1/4$ and $c_2 = 1/2$ resembles the above Crank--Nicolson scheme.
	
	Note further that also higher-order finite elements could be considered. This would only introduce a dependence on the polynomial degree, but the main arguments follow along the same lines.
\end{remark}

\subsection{Choice of parameters and computational complexity}\label{ss:choiceparams}

Finally, we would like to comment on the scaling of $\ell$ as specified in~\eqref{eq:choiceEll} and corresponding possible/optimal parameter choices for our method. 
To start with, the second and third term in~\eqref{eq:choiceEll} only include a moderate logarithmic scaling (provided that $\vartheta \eqsim h + \tau^2$) and are therefore typically much smaller than the other two terms. 
Moreover, for $T \eqsim H$ 
and with the realistic choice $\tau \geq h$, the first term in~\eqref{eq:choiceEll} scales like
\begin{equation*}
	C_{\tau,h} \tfrac T\tau \log\big(C_{\tau,h}\tfrac T\tau\big) \eqsim \tfrac \tau h \tfrac H\tau \log\big(\tfrac \tau h \tfrac H\tau\big) \eqsim \tfrac Hh \log\big(\tfrac Hh\big).
\end{equation*}
Note that a linear scaling with $H/h$ is physically reasonable to capture the wave cone, such that the first term includes only a logarithmic overhead. 
In particular, we have a reasonable scaling with respect to the spatial parameters independently of the choice of $\tau$. The last term in~\eqref{eq:choiceEll}, however, is still scaled by $\tau/h$ such that a variable scaling between $\tau$ and $h$ influences the necessary choice of~$\ell$, see also the examples in Section~\ref{ss:varscaling} below. There, we also present two examples where the first term is dominant (which is the case for $h \leq H^2$). In that setting, the scaling of the last term in~\eqref{eq:choiceEll} is not critical even for larger choices of $\tau \geq h$.

We summarize some optimal choices of the involved parameters, which are also observed in the numerical examples below:
\begin{itemize}
	\item due to the better scaling in time of the Crank--Nicolson scheme, $\tau \geq h$ is reasonable,
	\item for $H$ rather close to $h$, the choice of $\tau$ has an influence on the decay behavior, which can to some extend be compensated for by a smaller choice of $T$ (reset time),
	\item for $h \leq H^2$, the choice of $\tau$ has no severe influence and $T = H/\beta$ is an appropriate reset time that takes into account faster propagation for larger values of $\beta$,
	\item experimentally, $\ell \geq C \,H/h$ for some $ C \geq 2$ appears to be sufficient in
	the regime under consideration
	(cf., e.g., Figure~\ref{fig:exsol}) when $h \ll H$, although the theory predicts a logarithmic overhead.
\end{itemize}

Finally, we comment on the computational overhead introduced by our strategy. Without leveraging the fact that the computations on the subdomains are parallel, the overhead is mainly dependent on the number of overlapping subdomains. 
It scales like $(\ell\,h/H)^d$, and $\ell\,h/H$ is the number of coarse layers that are added (it makes sense here to count with respect to the coarse mesh on which the patches are defined). The advantage of our strategy, however, is that global computations are avoided and that the computational overhead can be easily compensated for by parallel computations with an appropriate computer. 
The computational overhead and the possibility to parallelize computations is also a feature that our approach shares with classical domain decomposition methods, see, e.g., \cite{TosW05,Gan08}, and particularly \cite{GanHN03,GanH05,GanKM21} in the context of the wave equation. However, domain decomposition approaches may use much smaller overlaps between the subdomains and require specifically chosen boundary conditions. Moreover, they typically involve multiple iterations. Our approach is non-iterative, but needs larger overlaps. A very attractive feature is also that it is very easy to implement and the subdomains only communicate at fixed reset times without the need for specific boundary conditions. 
A more in-depth comparison of the different approaches is subject to future research.   
\begin{figure} 
	\centering
	\scalebox{0.73}{
		% This file was created by tikzplotlib v0.9.6.
\begin{tikzpicture}

\begin{axis}[
legend cell align={left},
legend style={fill opacity=0.8, draw opacity=1, text opacity=1, at={(0.03,0.97)}, anchor=north west, draw=white!80!black},
legend pos = north east,
log basis x={10},
log basis y={10},
tick align=outside,
tick pos=left,
x grid style={white!69.0196078431373!black},
xlabel={\phantom{p} mesh size \(\displaystyle H\) \phantom{$\ell$}},
xmin=0.028, xmax=0.54,
xmode=log,
xtick style={color=black},
xtick = {0.03162277660168379,1e-1,0.31622776601683794},
y grid style={white!69.0196078431373!black},
ylabel={rel.~error},
ymin=1e-14, ymax=10,
ymode=log,
ytick style={color=black},
xminorticks=true,
]
\addplot [thick, myBlue, mark=square, mark size=3, mark options={solid}]
table {%
	0.5 4.354633228275700e-14
	0.25 1.091068178591713e-13
	0.125 2.051896136978938e-07
	0.0625 7.127215743682120e-04
	0.03125 1.256621546489939e-01
};
\addlegendentry{\scriptsize $h = 2^{-6}$}
\addplot [semithick, myRed, mark=diamond, mark size=4, mark options={solid}]
table {%
0.5 1.256946063154046e-13
0.25 1.089596689806574e-13
0.125 6.129021711494427e-13
0.0625 9.749583564420689e-07
0.03125 2.995454340817482e-03
};
\addlegendentry{\scriptsize $h = 2^{-7}$}
\addplot [semithick, myOrange, mark=triangle, mark size=4, mark options={solid}]
table {%
0.5 3.545457127017946e-13
0.25 3.051558244464410e-13
0.125 2.822090610031586e-13
0.0625 2.397297656464281e-12
0.03125 3.258703932184376e-06
};
\addlegendentry{\scriptsize $h = 2^{-8}$}
\addplot [semithick, myGreen, mark=o, mark size=3, mark options={solid}]
table {%
0.5 1.001297865572653e-12
0.25 8.632608216935513e-13
0.125 8.044291935583835e-13
0.0625 7.740327285582746e-13
0.03125 9.859718326891644e-12
};
\addlegendentry{\scriptsize $h = 2^{-9}$}
\addplot [semithick,dashed, myGreen]
table {%
		0.55 0.001953125
		0.015625 0.001953125
	};
%\addplot [very thick,dashed, myGreen]
%table {%
%		0.5 3.814697265625e-06
%		0.015625 3.814697265625e-06
%	};

%\addplot [semithick,dashed, myOrange]
%table {%
%	0.5 0.00390625
%	0.015625 0.00390625
%};
%\addplot [semithick,dashed, myOrange]
%table {%
%	0.5 1.52587890625e-05
%	0.015625 1.52587890625e-05
%};
%\addplot [semithick,dashed, myRed]
%table {%
%	0.5 0.0078125
%	0.015625 0.0078125
%};
%\addplot [semithick,dashed, myRed]
%table {%
%	0.5 6.103515625e-05
%	0.015625 6.103515625e-05
%};
%\addplot [semithick,dashed, myBlue]
%table {%
%	0.5 0.015625
%	0.015625 0.015625
%};
%\addplot [semithick,dashed, myBlue]
%table {%
%	0.5 0.000244140625
%	0.015625 0.000244140625
%};
%\addplot [semithick, black]
%table {%
%	0.5 0.5
%	0.25 0.25
%	0.125 0.125
%	0.0625 0.0625
%	0.03125 0.03125
%	0.015625 0.015625
%};
%\addlegendentry{\scriptsize order 1}
%\addplot [semithick, black, dashed]
%table {%
%0.5 0.25
%0.25 0.0625
%0.125 0.015625
%0.0625 0.00390625
%0.03125 0.0009765625
%0.015625 0.000244140625
%};
%\addlegendentry{\scriptsize order 2}
\end{axis}

\end{tikzpicture}}
	\scalebox{0.73}{
		% This file was created by tikzplotlib v0.9.6.
\begin{tikzpicture}

\begin{axis}[
legend cell align={left},
legend style={fill opacity=0.8, draw opacity=1, text opacity=1, at={(0.03,0.97)}, anchor=north west, draw=white!80!black},
legend pos = north east,
log basis x={10},
%log basis y={10},
tick align=outside,
tick pos=left,
x grid style={white!69.0196078431373!black},
xlabel={\phantom{p} parameter \(\displaystyle \ell\) \phantom{$\ell$}},
xmin=9.3, xmax=35.7,
%xmode=log,
xtick style={color=black},
xtick = {16,24,32},
y grid style={white!69.0196078431373!black},
ylabel={rel.~error},
ymin=1e-14, ymax=10,
ymode=log,
ytick style={color=black},
xminorticks=true,
]
\addplot [thick, myOrange, mark=triangle, mark size=4, mark options={solid}]
table {%
	10 5.059975466082818e+00
	11 3.640427800459589e+00
	12 4.556715924993481e+00
	13 3.128910896837699e+00
	14 1.408158978681319e+00
	15 1.434981006533453e+00
	16 1.084747656509911e+00
	17 4.280500353480426e-01
	18 1.161606182334856e-01
	19 2.934031931190369e-02
	20 7.057984387823395e-03
	21 1.520610207254769e-03
	22 3.248791029502339e-04
	23 6.701524971759793e-05
	24 1.206702464882849e-05
	25 1.776743585947019e-06
	26 2.429282074624669e-07
	27 3.578185450034470e-08
	28 5.510512150635897e-09
	29 8.101570483481841e-10
	30 1.207412837813102e-10
	31 1.778481046497086e-11
	32 2.397297656464281e-12
	33 3.946056746142811e-13
	34 2.764790016972873e-13
	35 2.725416767644139e-13
};
\addlegendentry{\scriptsize $\ell$-patch}
\addplot [semithick,dashed, black]
table {%
	16 100
	16 1e-15
};
\addplot [semithick,dashed, black]
table {%
	32 100
	32 1e-15
};
\end{axis}

\end{tikzpicture}}
	\caption{Errors between the local superposition method and the global Crank--Nicolson method for $A \equiv 1$; %. and different parameters; 
		left: with respect to $H$ for $\ell = 2H/h$ and variable~$h$; right: with respect to~$\ell$ for $h = 2^{-8}$ and $H = 2^{-4}$.}\label{fig:constA}
\end{figure}
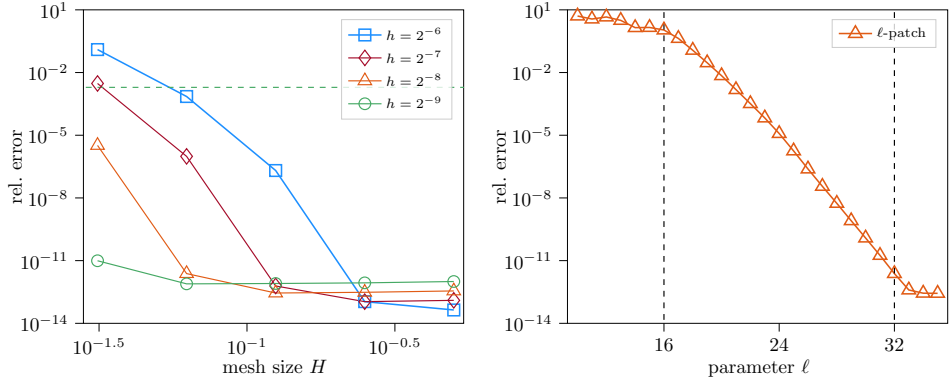

\section{Numerical examples}\label{s:numerics}

In this section, we illustrate the applicability of the local superposition method with a set of numerical examples. We measure all occurring errors in a discrete $L^2(0,t_\mathrm{fin}; H^1_0(\Omega))$-norm. More precisely, for the discrete function $v_h := (v_h^k)_{k \leq N_T} \in V_h$ evaluated at the coarse time points $kT$, we define
\begin{equation*}
	%\textstyle
	\|v_h\|_{h,T}^2 := \sum_{k \leq N_T} T\,\|v_h^k\|_a^2.
\end{equation*}
All computations are based on Matlab. To better compare the results of the different settings, we set $\Omega = (0,1)^2$, fix the right-hand side $f \equiv 1$, and choose zero initial conditions.
Further, we set $t_\mathrm{fin} = 1$. 
Note that in the examples we refer to the edge length of an element by $H$ and $h$, respectively (instead of the diameter) to avoid explicitly writing the additional factor $\sqrt{2}$.

\subsection{Constant coefficient and equal scaling}

For the first experiment, we choose a rather optimal setting regarding the scaling of the decay constant $\gamma$ in Lemma~\ref{lem:decay}. That is, $A \equiv 1$, $\tau = h$, and $T = H$. Figure~\ref{fig:constA} (left) shows the relative errors in $\|\cdot\|_{h,T}$ between the classical global Crank--Nicolson scheme and the local superposition method for the fixed choice $\ell = 2H/h$. As mentioned above, the first term in~\eqref{eq:choiceEll} is only dominant if $h \leq H^2$. This is also observed in the figure as the errors deteriorate for the fixed choice of $\ell$ when $H$ approaches $h$. For the case $h = 2^{-9}$ ({\protect \tikz{ \draw[line width=1.5pt, myGreen] circle (0.6ex);}}), we do not observe a deterioration of the error, which is significantly smaller than the error of the classical Crank--Nicolson scheme (the scaling $h=2^{-9}$ is indicated by the green dashed line). 
In particular, the moderate choice $\ell = 2H/h$, which corresponds to an increase of the local supports $\omega_i$ by two layers of coarse elements, is sufficient to achieve an almost identical solution with our localized scheme. 
In Figure~\ref{fig:constA} (right), the influence of the choice of $\ell$ is depicted for fixed $h = 2^{-8}$ and $H = 2^{-4}$. We observe that the error does not improve if $\ell \leq H/h$ which is in line with the localization result in Theorem~\ref{thm:locerr}, where the polynomial pre-factor in $\ell$ is dominant compared to the exponential term for smaller choices of $\ell$. From a physical point of view, this effect is also reasonable, because $H/h$ layers are required to capture the wave cone in the first place. For $\ell \geq H/h$, we observe a rapid exponential improvement of the error up to $\ell \approx 2H/h$, where the error does not improve anymore (at a relative error of the order $10^{-12}$). 
Note that $\ell = H/h$ and $\ell = 2H/h$ are highlighted by the vertical dashed gray lines.

\subsection{Constant coefficient and variable scaling}\label{ss:varscaling}
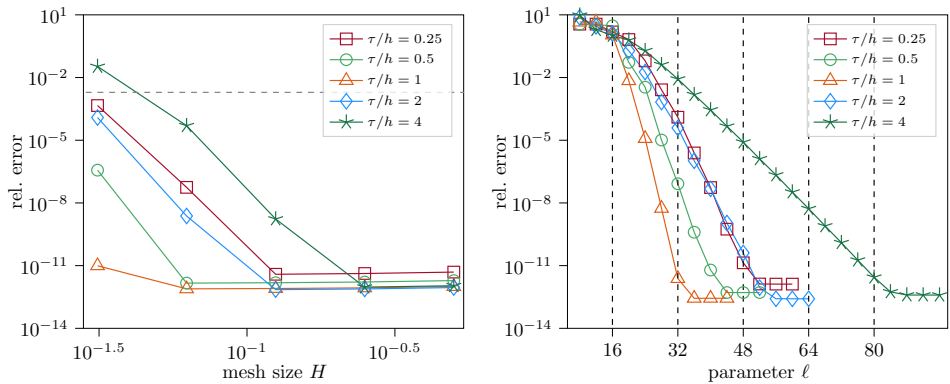
\begin{figure} 
	\centering
	\scalebox{0.73}{
		% This file was created by tikzplotlib v0.9.6.
\begin{tikzpicture}

\begin{axis}[
legend cell align={left},
legend style={fill opacity=0.8, draw opacity=1, text opacity=1, at={(0.03,0.97)}, anchor=north west, draw=white!80!black},
legend pos = north east,
log basis x={10},
log basis y={10},
tick align=outside,
tick pos=left,
x grid style={white!69.0196078431373!black},
xlabel={\phantom{p} mesh size \(\displaystyle H\) \phantom{$\ell$}},
xmin=0.028, xmax=0.54,
xmode=log,
xtick style={color=black},
xtick = {0.03162277660168379,1e-1,0.31622776601683794},
y grid style={white!69.0196078431373!black},
ylabel={rel.~error},
ymin=1e-14, ymax=10,
ymode=log,
ytick style={color=black},
xminorticks=true,
]
\addplot [semithick, myRed, mark=square, mark size=3, mark options={solid}]
table {%
0.5 4.904005270487529e-12
0.25 4.198100575457950e-12
0.125 3.859315323851530e-12
0.0625 5.560035192765970e-08
0.03125 4.572666594039695e-04
};
\addlegendentry{\scriptsize $\tau/h = 0.25$}
\addplot [semithick, myGreen, mark=o, mark size=3, mark options={solid}]
table {%
0.5 1.923099731940798e-12
0.25 1.648724381568709e-12
0.125 1.518736920870212e-12
0.0625 1.460506323889361e-12
0.03125 3.696668356382642e-07
};
\addlegendentry{\scriptsize $\tau/h = 0.5$}
\addplot [semithick, myOrange, mark=triangle, mark size=4, mark options={solid}]
table {%
0.5 1.001297865572653e-12
0.25 8.632608216935513e-13
0.125 8.044291935583835e-13
0.0625 7.740327285582746e-13
0.03125 9.859718326891644e-12
};
\addlegendentry{\scriptsize $\tau/h = 1$}
\addplot [semithick, myBlue, mark=diamond, mark size=4, mark options={solid}]
table {%
	0.5 8.939755314479448e-13
	0.25 7.659191845190523e-13
	0.125 7.090773366237204e-13
	0.0625 2.369637171280145e-09
	0.03125 1.217393855788684e-04
};
\addlegendentry{\scriptsize $\tau/h = 2$}
\addplot [semithick, darkgreen, mark=star, mark size=4, mark options={solid}]
table {%
	0.5 1.090765088754748e-12
	0.25 9.398970322888069e-13
	0.125 1.700355047887529e-09
	0.0625 4.964436663674147e-05
	0.03125 3.392995879755367e-02
};
\addlegendentry{\scriptsize $\tau/h = 4$}
\addplot [semithick,dashed, gray]
table {%
		0.55 0.001953125
		0.015625 0.001953125
	};
\end{axis}

\end{tikzpicture}}
	\scalebox{0.73}{
		% This file was created by tikzplotlib v0.9.6.
\begin{tikzpicture}

\begin{axis}[
legend cell align={left},
legend style={fill opacity=0.8, draw opacity=1, text opacity=1, at={(0.03,0.97)}, anchor=north west, draw=white!80!black},
legend pos = north east,
log basis x={10},
tick align=outside,
tick pos=left,
x grid style={white!69.0196078431373!black},
xlabel={\phantom{p} parameter \(\displaystyle \ell\) \phantom{$\ell$}},
xmin=5, xmax=98,
xtick style={color=black},
xtick = {16,32,48,64,80},
y grid style={white!69.0196078431373!black},
ylabel={rel.~error},
ymin=1e-14, ymax=10,
ymode=log,
ytick style={color=black},
xminorticks=true,
]
\addplot [semithick, myRed, mark=square, mark size=3, mark options={solid}]
table {%
	8  3.587442789658395e+00
	12 3.605576988644191e+00
	16 1.520413843181296e+00
	20 6.590364771996300e-01
	24 6.280411500830557e-02
	28 2.590321452243266e-03
	32 1.259904184093793e-04
	36 2.464129103398856e-06
	40 5.438385074364328e-08
	44 5.555639608674714e-10
	48 1.364490771406789e-11
	52 1.320924532491350e-12
	56 1.306212756262642e-12
	60 1.303729832619504e-12
};
\addlegendentry{\scriptsize $\tau/h = 0.25$}
\addplot [semithick, myGreen, mark=o, mark size=3, mark options={solid}]
table {%
	8  3.535735772276195e+00
	12 3.376149120588484e+00
	16 2.980558416431594e+00
	20 5.324412023612171e-02
	24 3.452570547016771e-03
	28 1.024164026941105e-05
	32 8.260925080618248e-08
	36 3.977965359189453e-10
	40 6.165403193374082e-12
	44 5.156091782805516e-13
	48 5.141822066176915e-13
	52 5.162735309498190e-13
};
\addlegendentry{\scriptsize $\tau/h = 0.5$}
\addplot [semithick, myOrange, mark=triangle, mark size=4, mark options={solid}]
table {%
	8  4.497553748748031e+00
	12 4.556715924993481e+00
	16 1.084747656509911e+00
	20 7.057984387823395e-03
	24 1.206702464882849e-05
	28 5.510512150635897e-09
	32 2.397297656464281e-12
	36 2.737849608966819e-13
	40 2.745272079025041e-13
	44 2.742237404305053e-13
};
\addlegendentry{\scriptsize $\tau/h = 1$}
\addplot [semithick, myBlue, mark=diamond, mark size=4, mark options={solid}]
table {%
	8  7.579482276917640e+00
	12 2.819999288160662e+00
	16 1.282511779266951e+00
	20 1.897062224915157e-01
	24 1.836970927637610e-02
	28 6.696322580573020e-04
	32 3.883021010677106e-05
	36 1.041527331801855e-06
	40 4.718807787358884e-08
	44 1.091032803617485e-09
	48 4.050336978847896e-11
	52 8.852704698644445e-13
	56 2.602425850839374e-13
	60 2.592719466218290e-13
	64 2.589978861351393e-13
};
\addlegendentry{\scriptsize $\tau/h = 2$}
\addplot [semithick, darkgreen, mark=star, mark size=4, mark options={solid}]
table {%
	8  9.561444468162057e+00
	12 1.953621556818047e+00
	16 1.035877696316409e+00
	20 6.157794180428411e-01
	24 1.924747539245219e-01
	28 4.124908518331889e-02
	32 8.472415051282201e-03
	36 1.515047346367775e-03
	40 2.822842747454476e-04
	44 4.669464594647855e-05
	48 8.146043459807381e-06
	52 1.290110855485905e-06
	56 2.169755777140957e-07
	60 3.308649341256450e-08
	64 5.305176541299102e-09
	68 7.940673735429245e-10
	72 1.253562596369311e-10
	76 1.831929675801508e-11
	80 2.789678766506408e-12
	84 5.539590925693087e-13
	88 3.925092480870499e-13
	92 3.912724376471554e-13
	96 3.896917807415525e-13
};
\addlegendentry{\scriptsize $\tau/h = 4$}
\addplot [semithick,dashed, black]
table {%
	16 100
	16 1e-15
};
\addplot [semithick,dashed, black]
table {%
	32 100
	32 1e-15
};
\addplot [semithick,dashed, black]
table {%
	48 100
	48 1e-15
};
\addplot [semithick,dashed, black]
table {%
	64 100
	64 1e-15
};
\addplot [semithick,dashed, black]
table {%
	80 100
	80 1e-15
};
\end{axis}

\end{tikzpicture}}
	\caption{Errors between the local superposition method and the global Crank--Nicolson method for $A \equiv 1$; 
		left: with respect to~$H$ for $\ell = 2H/h$ and $h = 2^{-9}$, and variable~$\tau/h$; right: with respect to~$\ell$ for $h = 2^{-8}$, $H = 2^{-4}$, and variable~$\tau/h$.}\label{fig:cstVar}
\end{figure}

In a second experiment, we investigate the influence of the scaling between $\tau$ and the mesh parameter $h$. First, we choose $A \equiv 1$, $h=2^{-9}$, and multiple values of $\tau$ and keep an equal scaling $T = H$. We also choose $\ell = 2H/h$ as above. For different values of $H$, the results are presented in Figure~\ref{fig:cstVar} (left). As expected, we observe a deterioration of the errors if $\tau/h$ is increased or decreased away from 1 since $\ell$ needs to be suitably adapted according to~\eqref{eq:choiceEll} (in particular the scaling of the last term). The size of $h$ is once again indicated by the dashed line. Note that for larger values of $H$, the effect is less influential because the choice of $\ell$ leads to (almost) global computations for coarse values of $H$. 
In Figure~\ref{fig:cstVar} (right), the decay behavior with respect to $\ell$ is presented for the fixed choices $h = 2^{-8}$, $H = 2^{-4}$, $T = H$, and multiple values of $\tau$. 
The orange line~({\protect \tikz{ \draw[line width=1.5pt, myOrange] (0,0) -- (0.13,0.23) -- (0.26, 0) -- cycle;}}) corresponds to the decay plot in Figure~\ref{fig:constA}~(right), where $\tau = h$. 
Note that only every fourth value of $\ell$ is plotted and multiples of $H/h$ are once again highlighted by vertical dashed gray lines.
The plot shows a slower decay behavior for $\tau/h \neq 1$, where the influence of $\tau < h$ is less significant than the influence of $\tau > h$. 
That is, the theoretically predicted scaling of $\ell$ (cf.~\eqref{eq:choiceEll}) seems to be pessimistic for $\tau < h$ as the plots do not even indicate a linear scaling of $\ell$ with respect to $\tau/h$ in that case. 
Nevertheless, to avoid larger patches due to an influence on $\ell$ one can use a smaller parameter $T$, which, in turn, increased the number of resets in our algorithm. Any choice $T/H < 1$ reduces the requirements on $\ell$ in the practical regime where the first term in~\eqref{eq:choiceEll} is dominant compared to the last term. The effect is illustrated for the case of $\tau/h = 0.25$ and $T/H \in \{0.25,0.5,1\}$ as well as $\tau/h = 2$ and $T/H \in \{0.5,1\}$ in Figure~\ref{fig:AVar}~(left). The case $T/H = 1$ ({\protect \tikz{ \draw[line width=1.5pt, myRed] (0,0) rectangle (0.23,0.23);}} and {\protect \tikz{ \draw[line width=1.5pt, myBlue] (0,0.1) -- (0.1,0.22) -- (0.2,0.1) -- (0.1,-0.02) -- cycle;}}, solid lines) corresponds to the respective lines in Figure~\ref{fig:cstVar}~(right). It can be observed that the adjustment of $T/H$ has a positive effect on the exponential decay. 
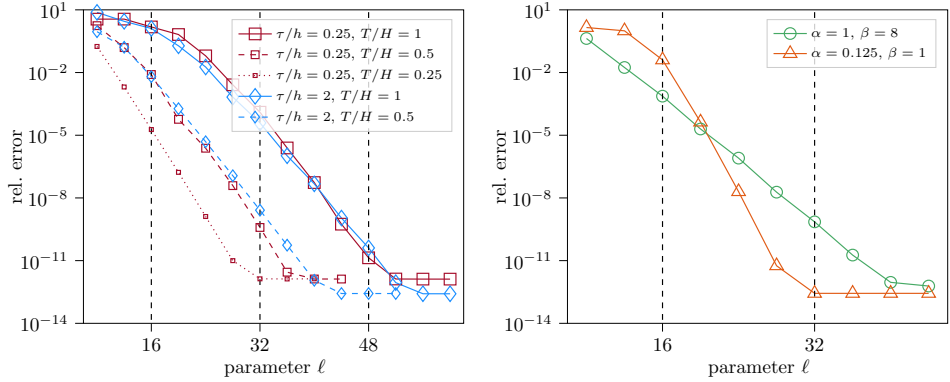
\begin{figure} 
	\centering
	\scalebox{0.73}{
		% This file was created by tikzplotlib v0.9.6.
\begin{tikzpicture}

\begin{axis}[
legend cell align={left},
legend style={fill opacity=0.8, draw opacity=1, text opacity=1, at={(0.03,0.97)}, anchor=north west, draw=white!80!black},
legend pos = north east,
log basis x={10},
tick align=outside,
tick pos=left,
x grid style={white!69.0196078431373!black},
xlabel={\phantom{p} parameter \(\displaystyle \ell\) \phantom{$\ell$}},
xmin=6, xmax=62,
xtick style={color=black},
xtick = {16,32,48,64,80},
y grid style={white!69.0196078431373!black},
ylabel={rel.~error},
ymin=1e-14, ymax=10,
ymode=log,
ytick style={color=black},
xminorticks=true,
]
\addplot [semithick, myRed, mark=square, mark size=3, mark options={solid}]
table {%
	8  3.587442789658395e+00
	12 3.605576988644191e+00
	16 1.520413843181296e+00
	20 6.590364771996300e-01
	24 6.280411500830557e-02
	28 2.590321452243266e-03
	32 1.259904184093793e-04
	36 2.464129103398856e-06
	40 5.438385074364328e-08
	44 5.555639608674714e-10
	48 1.364490771406789e-11
	52 1.320924532491350e-12
	56 1.306212756262642e-12
	60 1.303729832619504e-12
};
\addlegendentry{\scriptsize $\tau/h = 0.25$, $T/H = 1$}
\addplot [semithick, myRed, mark=square, mark size=2,dashed, mark options={solid}]
table {%
	8  1.750809343104200e+00
	12 1.563624903488529e-01
	16 8.036793930682391e-03
	20 5.775055007464105e-05
	24 2.392042154849075e-06
	28 3.884398768146736e-08
	32 3.867247201606693e-10
	36 2.801709708606280e-12
	40 1.321050343970743e-12
	44 1.315555658504075e-12
};
\addlegendentry{\scriptsize $\tau/h = 0.25$, $T/H = 0.5$}
\addplot [semithick, myRed, mark=square, mark size=1,dotted, mark options={solid}]
table {%
	8  1.769272299367761e-01
	12 2.035785419960714e-03
	16 1.873543864739934e-05
	20 1.675313239358839e-07
	24 1.291325395490702e-09
	28 1.000369789220608e-11
	32 1.347100138958966e-12
	36 1.329176808950391e-12
	40 1.349536661283045e-12
};
\addlegendentry{\scriptsize $\tau/h = 0.25$, $T/H = 0.25$}
\addplot [semithick, myBlue, mark=diamond, mark size=4, mark options={solid}]
table {%
	8  7.579482276917640e+00
	12 2.819999288160662e+00
	16 1.282511779266951e+00
	20 1.897062224915157e-01
	24 1.836970927637610e-02
	28 6.696322580573020e-04
	32 3.883021010677106e-05
	36 1.041527331801855e-06
	40 4.718807787358884e-08
	44 1.091032803617485e-09
	48 4.050336978847896e-11
	52 8.852704698644445e-13
	56 2.602425850839374e-13
	60 2.592719466218290e-13
};
\addlegendentry{\scriptsize $\tau/h = 2$, $T/H = 1$}
\addplot [semithick, myBlue, mark=diamond, mark size=3,dashed, mark options={solid}]
table {%
	8  9.007772119940809e-01
	12 1.618535842729734e-01
	16 6.324324633192595e-03
	20 1.864885354250369e-04
	24 4.876698272525369e-06
	28 1.155396089402710e-07
	32 2.587891895566900e-09
	36 5.481667706784674e-11
	40 1.147728319081424e-12
	44 2.667298035185099e-13
	48 2.667050271971610e-13
	52 2.663192138245387e-13
};
\addlegendentry{\scriptsize $\tau/h = 2$, $T/H = 0.5$}
\addplot [semithick,dashed, black]
table {%
	16 100
	16 1e-15
};
\addplot [semithick,dashed, black]
table {%
	32 100
	32 1e-15
};
\addplot [semithick,dashed, black]
table {%
	48 100
	48 1e-15
};
\addplot [semithick,dashed, black]
table {%
	64 100
	64 1e-15
};
\addplot [semithick,dashed, black]
table {%
	80 100
	80 1e-15
};
\end{axis}

\end{tikzpicture}}
	\scalebox{0.73}{
		% This file was created by tikzplotlib v0.9.6.
\begin{tikzpicture}

\begin{axis}[
legend cell align={left},
legend style={fill opacity=0.8, draw opacity=1, text opacity=1, at={(0.03,0.97)}, anchor=north west, draw=white!80!black},
legend pos = north east,
log basis x={10},
tick align=outside,
tick pos=left,
x grid style={white!69.0196078431373!black},
xlabel={\phantom{p} parameter \(\displaystyle \ell\) \phantom{$\ell$}},
xmin=6, xmax=46,
xtick style={color=black},
xtick = {16,32,48,64,80},
y grid style={white!69.0196078431373!black},
ylabel={rel.~error},
ymin=1e-14, ymax=10,
ymode=log,
ytick style={color=black},
xminorticks=true,
]
\addplot [semithick, myGreen, mark=o, mark size=3, mark options={solid}]
table {%
	8  4.253301554690593e-01
	12 1.744443869503539e-02
	16 7.414625708610254e-04
	20 1.990570140262316e-05
	24 8.028048025284037e-07
	28 1.887344143961010e-08
	32 7.163055787639575e-10
	36 1.844365168122489e-11
	40 9.133707242533110e-13
	44 6.086638946735607e-13
};
\addlegendentry{\scriptsize $\alpha = 1,\, \beta = 8$}
\addplot [semithick, myOrange, mark=triangle, mark size=4, mark options={solid}]
table {%
	8  1.427296999301366e+00
	12 9.915375008732010e-01
	16 4.035727549959312e-02
	20 4.267362171599163e-05
	24 1.996121390364620e-08
	28 5.678736193175597e-12
	32 2.720123169294697e-13
	36 2.716177917986430e-13
	40 2.720335670802433e-13
	44 2.717302376507690e-13
};
\addlegendentry{\scriptsize $\alpha = 0.125,\,\beta = 1$}
\addplot [semithick,dashed, black]
table {%
	16 100
	16 1e-15
};
\addplot [semithick,dashed, black]
table {%
	32 100
	32 1e-15
};
\addplot [semithick,dashed, black]
table {%
	48 100
	48 1e-15
};
\addplot [semithick,dashed, black]
table {%
	64 100
	64 1e-15
};
\addplot [semithick,dashed, black]
table {%
	80 100
	80 1e-15
};
\end{axis}

\end{tikzpicture}}
	\caption{Errors between the local superposition method and the global Crank--Nicolson method for $h = 2^{-8}$, $H = 2^{-4}$ with respect to~$\ell$; 
		left: $A \equiv 1$ and variable choices of $\tau$ and $T$; right: oscillatory $A$, $\tau = h$, and $T = H/\beta$.}\label{fig:AVar}
\end{figure}

We also present two examples where $h$ and $H$ are not as close as in the above experiments. As predicted by the theory, for $h \leq H^2$, the first dominant term regarding the choice of $\ell$ in~\eqref{eq:choiceEll} is independent of the value of $\tau$ provided that $\tau \geq h$. Therefore, a change of the scaling between $\tau$ and $h$ has no severe influence on the decay behavior, which is illustrated in the following. First, we consider a two-dimensional example with $A \equiv 1$ and a right-hand side such that $u(x,t) = \sin(\pi x_1)\sin(\pi x_2)\sin(0.5\pi t)^2$ is the exact solution, choose $H = T = 2^{-4}$, $h = 2^{-11}$, and $\tau = 2^{-8}$ for different values of~$\ell$. The results are plotted in Figure~\ref{fig:exsol} (left). The gray vertical dashed lines once again indicate multiples of $H/h$ while the red horizontal dashed line indicates the size of $h$.
We observe that the decay with respect to $\ell$ starts for $\ell \approx H/h$ and stagnates at $\ell \approx 2H/h$ as for the above cases where $\tau = h$. Since we compare the result of our algorithm to the exact solution, the level of stagnation is larger, related to the choice of $h$. This is in line with our theoretical findings.
The results of a second similar example in one dimension with exact solution $u(x,t) = \sin(\pi x)\sin(0.5\pi t)^2$ is presented in Figure~\ref{fig:exsol} (right). Here, we choose $H = T = 2^{-4}$, $h = 2^{-16}$, $\tau = 2^{-8} = \sqrt{h}$, and once again plot different values of $\ell$. As before, the decaying behavior is observed from $\ell \approx H/h$ to $\ell \approx 2H/h$, although the factor $\tau/h$ is even larger than before.

\subsection{Variable coefficient and equal scaling}

Finally, we present some results for a variable coefficient. More precisely, we choose $A$ as piecewise constant on a mesh of scale $2^{-5}$ with randomly chosen values between $\alpha$ and $\beta$. In Figure~\ref{fig:AVar} (right), we present the decay rates for $\alpha = 1,\,\beta = 8$ and $\alpha = 0.125,\,\beta = 1$, respectively. As before, we choose $h = 2^{-8}$ and $H = 2^{-4}$. Recall the definition of $C_{\tau,h}$ in Lemma~\ref{lem:decay}, which is proportional to $\beta\alpha^{-1}$. Although $\ell$ in~\eqref{eq:choiceEll} depends on $C_{\tau,h}$, one observes that a smaller $\alpha$ actually does not have a negative effect on the decay behavior. Further, as discussed in Section~\ref{ss:choiceparams} $T = H/\beta$ is chosen. This choice makes physically sense to compensate for the larger wave speed and results in a similar decay rate as observed for the case $A \equiv 1$, see also Figure~\ref{fig:constA}. Both observations are generally in line with the expected physical behavior for larger or smaller wave speeds.
\begin{figure} 
	\centering
	\scalebox{0.73}{
\begin{tikzpicture}

\begin{axis}[
legend cell align={left},
legend style={fill opacity=0.8, draw opacity=1, text opacity=1, at={(0.03,0.97)}, anchor=north west, draw=white!80!black},
legend pos = north east,
log basis x={10},
tick align=outside,
tick pos=left,
x grid style={white!69.0196078431373!black},
xlabel={\phantom{p} parameter \(\displaystyle \ell\) \phantom{$\ell$}},
xmin=-9, xmax=331,
xtick style={color=black},
xtick = {64,128,192,256},
y grid style={white!69.0196078431373!black},
ylabel={rel.~error},
ymin=1e-9, ymax=15,
ymode=log,
ytick style={color=black},
xminorticks=true,
]
\addplot [semithick, myRed, mark=square, mark size=3, mark options={solid}]
  table[row sep=crcr]{%
1	0.4460711436912\\
33	1.22953216377195\\
65	2.0536278405325\\
97	1.8778696851845\\
129	0.630785816721261\\
161	0.0215081164649137\\
193	0.000264993579138005\\
225	3.54417079642642e-06\\
257	3.98707250852587e-06\\
289	3.98872166345923e-06\\
321	3.98872543443872e-06\\
};
\addlegendentry{$\tau/h = 8$}
\addplot [semithick,dashed, black]
table {%
	128 100
	128 1e-15
};
\addplot [semithick,dashed, black]
table {%
	256 100
	256 1e-15
};
\addplot [semithick,dashed, myRed]
table {%
	-10 0.00048828125
	333 0.00048828125
};
\end{axis}

\end{tikzpicture}%
}
	\scalebox{0.73}{
\pgfplotsset{scaled x ticks=false}
\begin{tikzpicture}

\begin{axis}[
legend cell align={left},
legend style={fill opacity=0.8, draw opacity=1, text opacity=1, at={(0.03,0.97)}, anchor=north west, draw=white!80!black},
legend pos = north east,
log basis x={10},
tick align=outside,
tick pos=left,
x grid style={white!69.0196078431373!black},
xlabel={\phantom{p} parameter \(\displaystyle \ell\) \phantom{$\ell$}},
xmin=-319, xmax=10561,
xtick style={color=black},
xtick = {2048,4096,6144,8192},
y grid style={white!69.0196078431373!black},
ylabel={rel.~error},
ymin=1e-9, ymax=15,
ymode=log,
ytick style={color=black},
xminorticks=true,
]
\addplot [semithick, myRed, mark=square, mark size=3, mark options={solid}]
  table[row sep=crcr]{%
1	0.575436207664421\\
1025	1.11730058378904\\
2049	3.5949388933921\\
3073	3.89790707117656\\
4097	8.10298257121071\\
5121	0.230309160418104\\
6145	0.00308821328955871\\
7169	1.17984809553544e-05\\
8193	5.21684876040978e-08\\
9217	3.44402284410424e-09\\
10241	3.44733429096453e-09\\
};
\addlegendentry{$\tau = h^{1/2}$}
\addplot [semithick,dashed, black]
table {%
	4096 100
	4096 1e-15
};
\addplot [semithick,dashed, black]
table {%
	8192 100
	8192 1e-15
};
\addplot [semithick,dashed, myRed]
table {%
	-320 0.0000152587890625
	10562 0.0000152587890625
};
\end{axis}

\end{tikzpicture}%
}
	\caption{Errors between the local superposition method and the exact solution in $2d$ for $h = 2^{-11}$, $\tau = 2^{-8}$, and $H = T = 2^{-4}$ (left) and in $1d$ for $h = 2^{-16}$, $\tau = 2^{-8}$, and $H = T = 2^{-4}$ (right) with respect to~$\ell$.}\label{fig:exsol}
\end{figure}
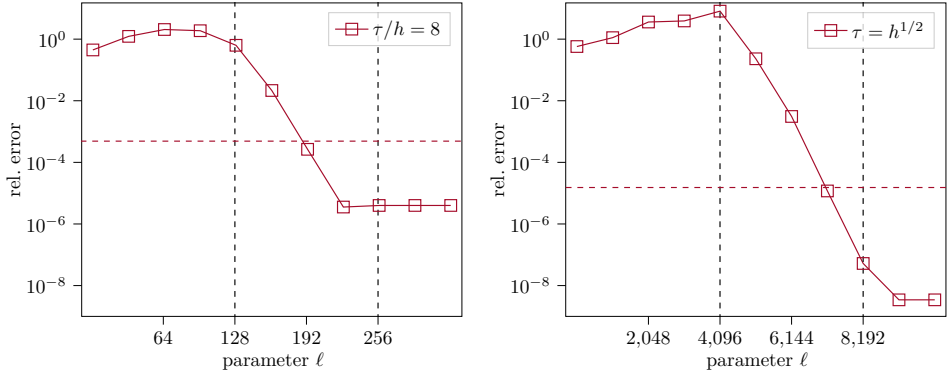

\section{Conclusions}

In this work, we have proposed a localized discretization strategy for the acoustic wave equation.
The idea is based on a superposition of local discrete solutions on overlapping subdomains using a combination of an implicit  Crank--Nicolson scheme and a first-order finite element method.
The localization is mathematically 
justified and evolves around the physical observation that waves travel with finite speed.
In particular, a physically reasonable overlap of the subdomains already allows one to well-approximate the globally defined Crank--Nicolson scheme. 
The algorithm may be understood as a domain decomposition strategy in space on successive short time intervals that does not require additional iterations or carefully designed boundary conditions. 
Moreover, parallelization with respect to the spatial variable is straightforward and
communication between the different subdomains is only required at a number of coarse time steps.  
The presented numerical experiments confirm the theoretical findings.

Our approach may be directly extended to other spatial discretizations and time discretization schemes such as a general Newmark method. Moreover, different discretization methods could be used on the respective subdomains, which would allow for a more flexible localized discretization. However, this requires a more involved analysis.

\end{document}